\documentclass[times,hidelinks]{elsarticle} 

\usepackage{amsmath,amssymb}
\usepackage{amsthm}
\usepackage{graphicx}
\usepackage{caption}
\usepackage{subcaption}
\usepackage[usenames, dvipsnames]{color}
\usepackage{csquotes}
\usepackage{verbatim}
\usepackage{bookmark}
\usepackage{epsfig}
\usepackage{url,hyperref}
\usepackage{longtable}
\usepackage{booktabs}
\usepackage{parskip} 
\usepackage{enumerate}

\usepackage[margin=1in]{geometry}

\theoremstyle{plain}
\newtheorem{theorem}{Theorem}
\newtheorem{lemma}{Lemma}

\newtheorem{example}{Example}

\theoremstyle{definition}

\theoremstyle{remark}
\newtheorem*{remark}{Remark}

\DeclareMathOperator{\tr}{tr}

\DeclareMathOperator{\ran}{Ran}

\DeclareMathOperator*{\argmin}{arg\,min}

\newcommand{\thmref}[1]{Theorem~\ref{#1}}
\newcommand{\lemref}[1]{Lemma~\ref{#1}}
\newcommand{\secref}[1]{Section \ref{#1}}
\newcommand{\figref}[1]{Figure~\ref{#1}}

\newcommand{\etc}{\textit{etc}.}

\newcommand{\ie}{\textit{i.e.}}
\newcommand{\eg}{\textit{e.g.}}

\newcommand{\ud}{\,\mathrm{d}}
\newcommand{\rd}{\mathrm{d}}

\newcommand{\CC}{\mathbb{C}}

\newcommand{\Int}{\mathbb{Z}}

\newcommand{\Complex}{\mathbb{C}}
\newcommand{\ee}{\mathbb{E}}

\newcommand{\reviewnumone}[1]{{#1}}
\newcommand{\reviewnumtwo}[1]{{#1}}
\newcommand{\addchange}[1]{{#1}}

\makeatletter
\newcommand*{\rom}[1]{\expandafter\@slowromancap\romannumeral #1@}
\makeatother
\newcommand{\vect}[1]{\boldsymbol{#1}}

\newcommand{\wt}[1]{\widetilde{#1}}

\IfFileExists{mathabx.sty}%
  {\DeclareFontFamily{U}{mathx}{\hyphenchar\font45}%
   \DeclareFontShape{U}{mathx}{m}{n}{<->mathx10}{}%
   \DeclareSymbolFont{mathx}{U}{mathx}{m}{n}%
   \DeclareFontSubstitution{U}{mathx}{m}{n}%
   \DeclareMathAccent{\widebar}{0}{mathx}{"73}%
}{%
  \PackageWarning{mathabx}{%
    Package mathabx not available, therefore\MessageBreak substituting
    widebar with overline\MessageBreak }%
  \newcommand{\widebar}[1]{\overline{#1}}%
}

\newcommand{\mc}[1]{\mathcal{#1}}

\newcommand{\eps}{\epsilon}

\newcommand{\abs}[1]{\lvert#1\rvert}
\newcommand{\Abs}[1]{\left\lvert#1\right\rvert}
\newcommand{\norm}[1]{\lVert#1\rVert}
\newcommand{\Norm}[1]{\left\lVert#1\right\rVert}

\newcommand{\Average}[1]{\left\langle#1\right\rangle}

\newcommand{\ket}[1]{\lvert#1\rangle}
\newcommand{\Ket}[1]{\left\lvert#1\right\rangle}

\newcommand{\inner}[2]{\langle#1, #2\rangle}
\newcommand{\Inner}[2]{\left\langle#1, #2\right\rangle}

\def\bigl{\mathopen\big}

\def\bigr{\mathclose\big}

 \newcommand\restr[2]{{
  \left.\kern-\nulldelimiterspace 
  #1 
  \vphantom{\big|} 
  \right|_{#2} 
  }}

 \newcommand{\dimn}{n}
 \newcommand{\hbt}{\Complex^{\dimn}}
\newcommand{\tm}{[0,T]}
\newcommand{\lr}{LR}
\newcommand{\lrt}{low-rank}

\newcommand{\mani}{\mathcal{M}}
\newcommand{\id}{\mathrm{Id}}

\newcommand{\gen}{\mathcal{A}}

\newcommand{\proj}{\mathcal{P}}
\newcommand{\projperp}{\mathcal{Q}}
\newcommand{\hs}{HS}
\newcommand{\lb}{Lindblad}
\newcommand{\fp}{Fokker-Planck}
\newcommand{\kl}{Karhunen-Lo{\` e}ve}

\newcommand{\lbop}{\mathcal{L}}
\newcommand{\gwieq}{Gr{\"o}nwall's inequality}
\newcommand{\ito}{It{\^o}}
\newcommand{\LU}{\wt{\vect{U}}}
\newcommand{\LSIG}{\wt{\vect{\sigma}}}
\newcommand{\sdm}{pseudometric}

\newcommand{\manitheta}{\mani_{\Complex^r}}

\newcommand{\funcsp}{\mathsf{F}}
\newcommand{\tang}{\mathcal{T}}
\newcommand{\uspace}{V_r(\hbt)}
\newcommand{\thismethod}{stochastic dynamical \lrt{} approximation}

\newcommand{\thismethodshort}{SDLR}
\newcommand{\hc}{\text{h.c.}}
\newcommand{\range}{\text{Ran}}
\newcommand{\opu}{\vect{U}}
\newcommand{\opg}{\vect{G}}
\newcommand{\opo}{\vect{O}}
\newcommand{\opv}{\vect{V}}
\newcommand{\opc}{\vect{C}}
\newcommand{\opd}{\vect{D}}
\newcommand{\oprho}{\vect{\rho}}
\newcommand{\ophami}{\vect{H}}
\newcommand{\oplb}{\vect{L}}
\newcommand{\opsigma}{\vect{\sigma}}
\newcommand{\opf}{\vect{F}}
\newcommand{\opeta}{\vect{\eta}}
\usepackage{csquotes}

\usepackage{mathtools}

\DeclarePairedDelimiter\floor{\lfloor}{\rfloor}

\newcommand{\Anticomm}[2]{\left[#1, #2\right]_{+}}
\newcommand{\Comm}[2]{\left[#1, #2\right]_{-}}

\graphicspath{{./Fig}}

\newcommand{\eqrefn}[1]{Eq. \eqref{#1}}

\newdimen\figrasterwd
\figrasterwd\textwidth

\begin{document}

\begin{frontmatter}

  \title{Stochastic dynamical low-rank approximation method}

  \author{Yu Cao} \address{Department of Mathematics, Duke University,
    Box 90320, Durham, NC 27708 USA} \ead{yucao@math.duke.edu}
  
  \author{Jianfeng Lu} \address{Departments of Mathematics, Physics,
    and Chemistry, Duke University, Box 90320, Durham, NC 27708 USA}
  \ead{jianfeng@math.duke.edu}
  
  \date{\today}

\begin{abstract}
  In this paper, we extend the dynamical low-rank approximation method
  to the space of finite signed measures.  Under this framework, we
  derive stochastic low-rank dynamics for stochastic differential
  equations (SDEs) coming from classical stochastic dynamics or
  unraveling of Lindblad quantum master equations. We justify the
  proposed method by error analysis and also numerical examples for
  applications in solving high-dimensional SDE, stochastic Burgers'
  equation, and high-dimensional \lb{} equation.
\end{abstract}


\end{frontmatter}


\section{Introduction}

\reviewnumone{Many problems in computational physics are challenging
  to solve due to \emph{curse of dimensionality}, such as high
  dimensional master equations and many-body quantum dynamics.  In
  attempt to resolve the difficulty, many ideas have been proposed:
  model reduction method
  \cite{antoulas_survey_2001,Chatterjee_00_pod}, Monte Carlo method
  \cite{Gisin92,Dalibard92,Tully}, \etc{} In many situations, several methods of
  dimension reduction need to be combined together.  For instance,
  after applying Monte Carlo method to some deterministic dynamics by
  simulating a stochastic differential equation (SDE) instead, the
  dimension of that SDE may still be very large.  Then it is
  attractive to further apply model reduction method in order to
  capture the main dynamical flows. This is our motivation to study
  model reduction method for high-dimensional SDEs arising from
  high-dimensional PDEs or matrix ODEs. In particular, our main
  motivation comes from two important physical systems: Fokker-Planck
  equation \cite{Pavliotis} and \lb{} equation
  \cite{Lindblad76,Gorini76,Breuer}.

  \fp{} equations and \lb{} equations are the governing master
  equations to describe system evolution for open classical and
  quantum systems respectively under Markovian approximation.  Both
  are challenging to solve when the dimension becomes large. To
  resolve this problem, it is standard to consider Monte Carlo
  (particle) method based on stochastic differential equations, with
  statistical average of sample trajectories to obtain the quantity of
  interest.  For the quantum case, such methods are known as
  \enquote{unraveling} and \enquote{stochastic wave-function method}
  for Lindblad equation \cite{Gisin92,Dalibard92}.  }

More specifically, suppose we would like to solve the Fokker-Planck
equation $\partial_t \mu_t = \gen_t \mu_t$ where $\mu_t$ is the
probability distribution (or measure to be more general) and $\gen_t$
is a time-parametrized operator mapping a probability distribution to
its tangent space.  In the particle based methods, one simulates a SDE
$X_t$ with infinitesimal generator $\gen_t^{*}$ (the adjoint operator
of $\gen_t$) and with initial condition $X_0$ drawn from $\mu_0$ and thus
the distribution of $X_t$ is exactly $\mu_t$.  Similarly, as the
quantum analog, \lb{} equation has the form
$\frac{\ud}{\ud t} \oprho = \lbop(\oprho)$ where $\oprho$ denotes the
density matrix of a quantum system and $\lbop$ is the generator for a
completely positive dynamical semigroup \cite{Lindblad76,Gorini76}. 
\reviewnumone{One may solve it by sampling a SDE
$X_t$ such that $\ee\bigl[X_t X^{\dagger}_t\bigr]$ is
exactly the solution of \lb{} equation (see Lemma
\ref{lemma::lb_unravel} below for more details).} 
\addchange{There are various
choices of the SDEs, \eg, quantum state diffusion (QSD) \cite{Gisin92} and linear
quantum state diffusion (LQSD) \cite{Brun00}.} While it is also possible to use other
stochastic processes such as jumping process \cite{Dalibard92, Brun00},
we will limit the scope of our consideration to Monte Carlo methods
based on diffusion processes.

The Monte Carlo method for both \fp{} equation and \lb{} equation can be described under the same framework:
\begin{displayquote}
  Given a $\hbt$-valued SDE $X_t$, one would like to approximate
  $\ee\left[f(X_t)\right] \equiv \int f\ \ud\mu_t$ for a collection of
  prescribed functions $f \in \funcsp$, where $\mu_t$ is the
  distribution of $X_t$.
\end{displayquote}
In the case of Fokker-Planck equation, $\funcsp$ could be a collection
of smooth functions; in the case of \lb{} equation, $\funcsp$ could be
a singleton set $\left\{ f(x) = x x^{\dagger} \right\}$ where
$x\in \hbt$. In the sequel, we shall consider $f(x) = x x^{\dagger}$
only, which turns out to be an interesting and useful choice: in the
case of Fokker-Planck equation, choosing such $f$ means one would like
to calculate the second moment of measure $\mu_t$; in the case of
\lb{} equation, choosing such $f$ means one would like to compute the
density matrix.

To reduce the computational complexity, 
a popular approach is
model reduction, that is, to retrieve the dynamics  by only
capturing the evolution of a lower-dimensional object.  For our case,
there are two directions:
\begin{enumerate}[(i)]
\item find a low-rank approximation for $X_t$,  or,
\item find a low-rank approximation for $\mu_t$.  
\end{enumerate}
\addchange{In the literature, there are several methods taking
the first approach, in the flavor of \kl{} expansion (KLE): for
instance, proper orthogonal method (POD) \cite{Chatterjee_00_pod, sapsis_dynamically_2011}, dynamical orthogonal (DO) method \cite{sapsis_dynamically_2011,Sapsis09, Sapsis12,
	Musharbash15} and dynamical bi-orthogonal method (DyBO)
\cite{Cheng13_theory, Cheng13_complexity}.
It is clear that in the above framework for our cases, the realization of randomness in $X_t$ is not important, whereas the distribution $\mu_t$ is the key for accurate approximation.} Hence, it is natural to consider the low-rank approximation for $\mu_t$, \ie, on the space of probability measures. Then the problem is formulated as follows:
\begin{displayquote}
\addchange{
  Given a collection of prescribed test functions $\funcsp$
  and the time evolution equation of probability measures
  $\partial_t \mu_t = \gen_t \mu_t$ on $\hbt$, 
  one would like
  to find low-rank approximation $\mu_{\lr,t}\approx \mu_t$ such that
\[\sup_{f\in \funcsp} \Norm{\int f\ \ud\mu_t - f\ \ud\mu_{\lr,t} }\]
is small.
}
\end{displayquote}
As a remark, in \secref{sec::sdlr}, we shall use the space of finite
signed measures, instead of probability measures to avoid the
technicality; please see the discussion in \secref{sec::sdlr} for
details.

Our work is motivated by extending the (deterministic) dynamical
low-rank approximation, introduced by Koch and Lubich in
\cite{Lubich07} for matrix ODEs, to the stochastic case. The main idea in the dynamical low-rank
approximation has been illustrated in the context of matrix ODE
\cite{Lubich07}, summarized in the next paragraph.

Consider a matrix ODE system $M(t)\in \Complex^{n\times n}$, 
\begin{equation}
  \frac{\ud}{\ud t}{M}(t) = F(t, M(t)).
\end{equation}
The dynamical low-rank approximation method in \cite{Lubich07} consists of two steps. Firstly,
identify a sub-manifold $\mathcal{M}_{r}\subset \Complex^{n\times n}$
and approximate the matrix ODE solution $M(t)$ by
$M_{\lr}(t)\in \mathcal{M}_r$ for all $t$; secondly, the time-evolution $M_{\lr}(t)$ is given by
\begin{equation}
\label{eqn::tsp_ode}
  \frac{\ud}{\ud t} M_{\lr}(t) = \argmin_{v \in \tang_{M_{\lr}(t)} \mc{M}_r} d\bigl(v, F(t, M_{\lr}(t))\bigr),
\end{equation}
where $\tang_{M_{\lr}(t)} \mc{M}_r$ is the tangent space of $\mathcal{M}_r$ at
the current location $M_{\lr}(t)$ and $d(v_1, v_2) := \Norm{v_1 - v_2}$ is a metric on tangent space; thus the evolution is constructed as
close as possible to the solution of matrix ODE by projecting $F(t, M_{\lr}(t))$ onto the tangent space $\tang_{M_{\lr}(t)}\mani_r$.

In our proposed method, we adopt this idea to the space of finite
signed measures on $\hbt$ with bounded second moment, denoted by $\mani$. The subspace $\mani_r$
in this case is defined as the space of finite signed measures
supported on a linear subspace of $\hbt$ with dimension at most
$r$. Then, we hope to approximate $\mu_t$ by $\mu_{\lr,t}\in
\mani_r$.
The time-evolution equation of the low-rank approximation is given by
\begin{equation*}
 \partial_t \mu_{\lr,t} \equiv \gen_{\lr,t} \mu_{\lr,t} := \argmin_{\nu \in \tang_{\mu_{\lr,t}}\mani_r}  d_{\funcsp}\left(\nu,\ \gen_{t} \mu_{\lr,t}\right),
 \end{equation*}
 where $d_{\funcsp}$ is a \sdm{} defined in \eqrefn{eqn::metric}
 below and $\tang_{\mu_{\lr,t}} \mani_r$ is the tangent space of
 $\mani_r$ at $\mu_{\lr,t}$. We will refer this method as the
 \thismethod{} method (or \thismethodshort{} in abbreviation).

As a concrete example, let $\funcsp = \left\{ f(x) = x x^{\dagger} \right\}$ be a singleton set, consisting only one test function (which maps to $\CC^{n\times n}$). Assume that the time evolution equation $\partial_t \mu_t  = \gen_{t} \mu_t$ is the Fokker-Planck equation of a SDE of the form
\begin{equation}
\label{eqn::sde}
\ud X_t = a(X_t, t)\ud t + \sum_{j=1}^{N} b_j(X_t, t) \ud W_j,
\end{equation}
where $X_t\in \hbt$, $a$ and $b_j$ are functions
$\hbt\times \tm \rightarrow \hbt$ and $W_j$ are independent
real-valued standard Brownian motions. 
With some additional assumptions \reviewnumtwo{and restrictions}, one could obtain the low-rank dynamics given by \eqrefn{eqn::low_rank_xx} (or equivalently \eqrefn{eqn::low_rank_xx_v2}). The details are given in \thmref{thm::fx=xx}, which is one of the main results in this paper.


As already mentioned above, the stochastic dynamical low-rank
approximation is also motivated by developing efficient methods for
the Lindblad quantum master equations. In that context, the
deterministic low-rank approximation has been studied by Le Bris and
Rouchon to find low-rank approximation of \lb{} equation
\cite{LeBris13}.  In the subsequent work \cite{LeBris15}, Le Bris,
Rouchon and Roussel also introduced an unraveling scheme for the
low-rank quantum master equation obtained in \cite{LeBris13}.  The
unraveling of Lindblad equations and its connection with the low-rank
approximation will be discussed in
Section~\ref{sec::discuss::unravel}. In particular, as another main
result of this paper, we establish a commuting diagram of unraveling
and low-rank approximation, with the proposed \thismethodshort{}
method.

The rest of the paper is organized as follows. In \secref{sec::sdlr},
we shall formulate the \thismethod{} method in the space of finite
signed measures. Then in \secref{sec::example}, we will provide a concrete
example, in which a low-rank dynamics for \fp{} equation is derived,
as well as the low-rank dynamics of the SDE for that \fp{}
equation. The comparison of our method and DO method will also be given at
the end of \secref{sec::example}.
\reviewnumtwo{
Consistency of our low-rank approximation and
error analysis will be provided in \secref{sec::error}.
In \secref{sec::discuss::unravel},
we shall establish the connection between the action of dynamical low-rank
approximation and the action of unraveling. We will prove a commuting
relation between them.}
Then numerical
results will be presented in \secref{sec::numerics} to demonstrate the
performance. In \secref{sec::conclusion}, we will give a brief summary
and some potential follow-up work.

Throughout this paper, 
$\Norm{\cdot}_{\hs}$ means Hilbert-Schmidt norm (or
Frobenius norm as in linear algebra). 
Bold letters, like $\opu$ and $\opg$ \etc, represent
matrices (linear operators). Also, $\Comm{\cdot}{\cdot}$ is commutator and 
$\Anticomm{\cdot}{\cdot}$ is anti-commutator.

\section{Stochastic dynamical low-rank approximation method}
\label{sec::sdlr}
As we recalled in the introduction, the dynamical low-rank
approximation method \cite{Lubich07}, developed for deterministic ODE
dynamics, involves the identification of an approximate sub-manifold
and projection onto the tangent space by solving a minimization
problem.  In this section, we will adopt this idea to formulate the
dynamical low-rank approximation in the space of finite signed
measures on $\hbt$ with bounded second moment.
This low-rank approximation method offers an abstract framework, for
instance, to approximate both \fp{} and \lb{} equations via low-rank
dynamics combined with the particle methods.  Thus the proposed method
will be named \emph{\thismethod{}} (\thismethodshort{}) method.  A
concrete example and corresponding low-rank dynamics will be given in
\secref{sec::example} below.

\subsection{Problem setup and low-rank approximation}

Consider the measure space $(\hbt, \mathcal{B})$, where $\mathcal{B}$
is the $\sigma$-algebra of Borel sets on $\hbt$.  Denote $\mani$ the
collection of finite signed measures with bounded second moment on this
measure space:
\begin{equation*}
  \mani := \biggl\{\, \mu \text{ is a finite signed measure}\ \bigg\vert \  \int \abs{x}^2\ \Abs{\mu}(\ud x)  < \infty \, \biggr\}\,,
\end{equation*}
where the positive measure $\Abs{\mu}$ is variation of measure $\mu$.

Consider a given differentiable trajectory $\mu_t\in \mani$ solving
\[\partial_t \mu_t = \gen_t \mu_t,\]
where $\gen_t: \mani \rightarrow \mani$ is a given time-dependent
(linear) operator. 
In the context of Fokker-Planck
equation, $\gen_t$ is the adjoint operator of the infinitesimal
generator of the corresponding SDE. In the context of \lb{} equation,
$\gen_t$ is the adjoint operator of the infinitesimal generator of the
SDE-type unraveling scheme of that \lb{} equation (see
\secref{sec::discuss::unravel}).

The low-rank approximation of $\mani$, denoted by $\mani_r$, is a
subset of $\mani$, which contains all measures in $\mani$ with support
on a $r$-dimensional linear subspace of $\hbt$.  Such low-rankness is
used to deal with the problem of high dimensionality of $\hbt$.  As a
remark, the low-rankness we explore here is not in the sense of taking
an ansatz of the measure in the space $\mani$ as a linear combination
of a few prescribed measures as a basis (which would be a usual
Galerkin approximation in the space of measures). Instead, the
low-rankness here means that the measure $\mu$ is mostly concentrated
on a $r$-dimensional linear subspace of $\hbt$, where $r \ll n$.
Intuitively, this approximation would work well for some dissipative
dynamics for which the measure is contracted to some low-dimensional
space as time evolves (see \secref{sec::numerics} for numerical
demonstration).

Let us characterize the structure of $\mani_r$. For any
$\mu_{\lr}\in \mani_r$, by definition, it is supported on a
$r$-dimensional linear subspace, whose orthonormal basis is denoted by
$\left\{U_1, U_2, \cdots, U_r \right\}$. Then one could define a
linear mapping $\opu: \Complex^r\rightarrow \hbt$ by
\begin{equation*}
\opu: y\in \Complex^r \rightarrow \sum_{j=1}^{r} U_j y_j, \qquad \text{ or in matrix form } \opu = \begin{bmatrix} U_1 & U_2 & \cdots & U_r \end{bmatrix}.
\end{equation*}
Let us denote the $r$-dimensional Stiefel manifold on $\hbt$ by
\begin{equation}
  \uspace :=  \left\{ \opu: \Complex^r \rightarrow \hbt \text{ is linear, and }\ \opu^{\dagger} \opu = 
 \id_{r\times r} \right\}.
\end{equation}
Then, $\mani_r$ could be viewed as a collection of
measures in $\mani$ with support on $\range(\opu)$ with bounded second moment, where $\range(\opu)$ is the range of some
linear operator $\opu \in \uspace$.  The restriction of $\mu_{\lr}$ on
$\text{Ran}(\opu)$ can be represented as a finite signed measure on
$\Complex^r$, given by the pullback
\[\theta(E) := \mu_{\lr}(\opu E) 
= \mu_{\lr}(x \in \range( \opu): x\in \opu E) = \mu_{\lr}(x \in \range( \opu): \opu^{\dagger} x\in E) 
\]
for any Borel set $E \subset \Complex^r$. 
Hence, for any measurable
set $F \subset \text{Ran}(\opu)\in \mathcal{B}$,
\begin{equation}
\mu_{\lr}(F) = \mu_{\lr} (\opu \opu^{\dagger} F) = \theta( \opu^{\dagger} F).
\end{equation}
Thus, there exists a one-to-one correspondence between $\mani_r$ and
$\uspace \oplus \manitheta$ where $\manitheta$ denotes the space of
finite signed measures on $\Complex^r$,
with bounded second moment. 
The low-rank dynamics
on $\mani_r$ that we shall consider is equivalent to the dynamics of a
pair $\bigl(\opu(t), \theta_t\bigr)\in \uspace \oplus \manitheta$.

We remark that we use in the general framework finite signed measure
instead of probability measure to avoid the subtleties arising from
the geometry of probability measures (due to the positivity), see for
example \cite{Ambrosio, Lott17}. In practice, we will guarantee that
the resulting dynamics yields probability measure by imposing more
constraints on the low-rank approximation, see \secref{sec::example}.

\subsection{Tangent space projection}

It is well-known that the tangent space of $\uspace$ at $\opu \in \uspace$ is given by (see for example \cite[Theorem 1.2]{Higham96}),
\begin{equation}\label{eq:tangUV}
\tang_{\opu} \uspace = \left\{ i \opg \opu: \opg^{\dagger} = \opg \text{ is a linear operator on }\hbt \right\}.
\end{equation}
A self-contained proof is provided in \ref{sec::stiefel_tangent} for readers' convenience. 
Thus for a
differentiable trajectory $\opu(t)\in \uspace $, we have
$\frac{\rd}{\rd t} \opu(t) = i \opg(t) \opu(t)$ for some Hermitian matrices $\opg(t)$.

\reviewnumtwo{Consider} any differentiable trajectory $\theta_t$ that 
\begin{equation}\label{eq:thetaeqn}
  \partial_t \theta_t = \gen_{\theta,t} \theta_t, 
\end{equation}
where $\gen_{\theta,t}: \manitheta\rightarrow \manitheta$ is some
operator.
The tangent space of $\mani_{r}$ at $(\opu(t), \theta_t)$ is fully characterized by
\begin{equation}
  \tang_{\opu(t)} \uspace \oplus \tang_{\theta_t} \manitheta.
\end{equation}

It is straightforward to adopt the idea of tangent space projection (\ie,  \eqrefn{eqn::tsp_ode}) to our situation. Consider a natural \sdm{} on $\mani$
\begin{equation}
\label{eqn::metric}
d_{\funcsp} \left(\nu_1, \nu_2\right):= \sup_{f\in \funcsp} \Norm{\int f\ \ud \nu_1 - f\ \ud \nu_2 },\qquad \nu_{1,2} \in  \mani.
\end{equation}
Recall that $\funcsp$ is a collection of test functions and
$\Norm{\cdot}$ is some suitable norm associated with functions in
$\funcsp$. In \secref{sec::example}, we will choose
\[\funcsp = \left\{f = xx^{\dagger}\right\}\] 
be a singleton and the norm $\Norm{\cdot}$ is chosen as
Hilbert-Schmidt norm. An equivalent choice is
that
\[\funcsp = \Bigl\{\,f(x) = \Inner{x}{\opo x} \;\big\vert \;
\text{Hermitian matrix } \opo \text{ satisfies } \norm{\opo}_{\hs} \le
1\Bigr\},\]
and the norm is simply the absolute value. The equivalence of these
two choices is proved in Lemma~\ref{lemma::choice_equivalent_F} in
\ref{sec:equivalence}.  In fact, from the perspective of quantum mechanics,
this equivalence is natural, since finding a good approximation of
density matrix $\int f\ \ud\nu \equiv \ee_{\nu}\bigl[xx^{\dagger}\bigr]$
is equivalent to finding a good approximation of all observations as
$\Average{\opo}_{\text{avg}} := \ee_{\nu} \bigl[ \Inner{x}{\opo
  x}\bigr] = \tr\left( \ee_{\nu} \bigl[ \opo xx^{\dagger}\bigr]\right)
= \Inner{\opo}{\ee_{\nu}\bigl[xx^{\dagger}\bigr]}_{\hs}$ where observable $\opo$ is a Hermitian matrix.

The tangent space projection of
$\partial_t \mu_t = \gen_{t} \mu_t$ to the tangent space
$\tang_{\mu_{\lr,t}} \mani_r$ is then given by 
\begin{equation}
\label{eqn::sdlr_minimization}
\begin{split}
  \partial_t \mu_{\lr,t} \equiv \gen_{\lr,t} \mu_{\lr,t} := & \argmin_{\nu \in \tang_{\mu_{\lr,t}}\mani_r}  d_{\funcsp}\left(\nu,\ \gen_{t} \mu_{\lr,t}\right) \\
= & \argmin_{\wt{\gen}_{\lr,t}:\ \wt{\gen}_{\lr,t} \mu_{\lr,t} \in \tang_{\mu_{\lr,t}} \mani_r} \sup_{f\in \funcsp} \Norm{\int f\ \ud \bigl(\wt{\gen}_{\lr,t}\mu_{\lr, t}\bigr)  - f\ \ud \bigl(\gen_t \mu_{\lr,t}\bigr) }. \\
\end{split}
\end{equation}
In the second line, the minimization problem is reformulated from finding tangent vector $\nu$ to finding differential operators $\wt{\gen}_{\lr,t}$. Though the notation is slightly abused, the variational problem above should still be clear.

Equivalently, using the adjoint operators $\wt{\gen}_{\lr,t}^{*}$ and $\gen_{t}^{*}$, it can be written as 
\begin{equation}
\label{eqn::sdlr_minimization_equiv}
\begin{split}
 \partial_t \mu_{\lr,t} = & \argmin_{\wt{\gen}_{\lr,t}^{*}:\ \wt{\gen}_{\lr,t} \mu_{\lr,t} \in \tang_{\mu_{\lr,t}} \mani_r} \sup_{f\in \funcsp} \Norm{\int \bigl(\wt{\gen}_{\lr,t}^{*} f\bigr)\ \ud \mu_{\lr, t} - \bigl(\gen_t^{*} f\bigr)\ \ud \mu_{\lr,t}} \\
=& \argmin_{\wt{\gen}_{\lr,t}^{*}:\ \wt{\gen}_{\lr,t} \mu_{\lr,t} \in \tang_{\mu_{\lr,t}} \mani_r} \sup_{f\in \funcsp} \Norm{\ee_{\mu_{\lr,t}} \left[\wt{\gen}_{\lr,t}^{*} f - \gen_t^{*} f \right] }. \\
\end{split}
\end{equation}
In \secref{sec::example}, we shall parametrize the infinitesimal
generator $\wt{\gen}_{\lr,t}^{*}$ by some functions to simplify the
minimization problem and also to avoid the vagueness of generic
infinitesimal generator for general stochastic processes.

\section{SDLR method for SDEs with $N$ driving Brownian motions and test function $f(x) = x x^{\dagger}$ }
\label{sec::example}

Based on the framework of SDLR, we may explore various low-rank
dynamics. 
In this section, $\mu_t$ is considered to be the probability measure of
$X_t$ given in \eqrefn{eqn::sde} and thus
$\partial_t \mu_t = \gen_t \mu_t $ corresponds to the Fokker-Planck
equation of SDE $X_t$ in \eqrefn{eqn::sde}. The test function space
$\funcsp$ is taken to be a singleton with the only element
$f(x) = x x^{\dagger}$. We will assume these in the sequel without
explicit mentioning.

To apply the method developed in last section, we impose some further
restrictions on $\tang_{\mu_{\lr,t}} \mani_r$ and thus on the choice
of low-rank dynamics. We will make some further comments
on these restrictions after we derive the resulting low-rank dynamics
by applying \thismethodshort{} method.
\begin{enumerate}
\item In the tangent space of $\uspace$, for Hermitian matrix $\opg(t)$ in
  \eqrefn{eq:tangUV}, we \reviewnumtwo{consider only those $\opg(t)$ such that}
  \begin{equation}
  \label{eqn::assmp_g}
  \opg(t) = \projperp_{\opu(t)} \opg(t) \proj_{\opu(t)} + \hc,
  \end{equation} where projection operator
  $\proj_{\opu(t)} := \opu(t) \opu^{\dagger}(t)$ and orthogonal projection operator
  $\projperp_{\opu(t)} := \id - \proj_{\opu(t)}$.  
\item In the tangent space of $\manitheta$, \reviewnumtwo{consider only those $\gen_{\theta,t}$ corresponding}
 to the Fokker-Planck equation of some
  SDE on $\Complex^r$ (cf.~\eqrefn{eq:thetaeqn}). 
  This basically means that we parametrize the infinitesimal generator
  $\gen_{\lr,t}^{*}$ (and hence $\gen_{\lr,t}$) by a collection of
  functions. The exact form will be given below in
  Lemma~\ref{lemma::generator}. Because the space of all possible
  infinitesimal generators is opaque and quite large, hence, choosing
  infinitesimal generators of a particular form is necessary in
  practice.
\end{enumerate}

Following the second \reviewnumtwo{constraint} above, let us denote $Y_t(\omega)$ the
corresponding SDE on $\Complex^r$ whose infinitesimal generator is
$\gen_{\theta,t}$. Consider the following family of stochastic
dynamics with $\opg$, $A$ and $B$'s to be chosen
\begin{equation}\label{eqn::sdeY}
\begin{split}
  \dot{\opu}(t) &=  i \opg(t) \opu(t),\qquad 
  \text{with } \opg^{\dagger}(t) = \opg(t),\  \opg(t) = \projperp_{\opu(t)} \opg(t) \proj_{\opu(t)} + \hc  \\
  \ud Y &= A(Y,t) \ud t + \sum_{j=1}^{M} B_j(Y,t) \ud W_j. \\
  \end{split}
\end{equation}
Define $X_{\lr, t}(\omega) := \opu(t) Y_t(\omega)$ and denote
$\mu_{\lr, t}$ the probability measure induced by random variable
$X_{\lr,t}(\omega)$. It is straightforward to check that
$X_{\lr,t}(\omega)$ satisfies the SDE
\begin{equation}
  \ud X_{\lr} = a_{\lr}\left(X_{\lr}, t\right)\ud t + \sum_{j=1}^{M}b_{\lr,j} \left(X_{\lr}, t\right)\ud W_j,
  \label{eqn::sde_lowrank}
\end{equation}
where, for $1\le j\le M$, 
\begin{equation}
  \begin{cases}
    \begin{aligned}
      a_{\lr}(X_{\lr,t}, t) & =  
      i \opg(t) X_{\lr,t} + \opu(t) A(\opu^{\dagger}(t) X_{\lr,t}, t),
    \end{aligned} \\
    b_{\lr, j}(X_{\lr,t}, t)= \opu(t) B_j(\opu^{\dagger}(t) X_{\lr,t}, t).
  \end{cases}
\end{equation}
Note that $M$ does not have to be the same as $N$ and $\rd W_j$ in
\eqrefn{eqn::sdeY} are not necessarily the same Brownian motions in
\eqrefn{eqn::sde}; we use $\rd W_j$ for both to save notation. Since
we are interested in the error in the weak sense, how the randomness is
achieved does not matter; what is important is the infinitesimal
generator which does not depend on the particular realization of the
Brownian motion.

\medskip

\begin{lemma}
\label{lemma::generator}
The adjoint infinitesimal generator, acting on $f(x) = x x^{\dagger}$,
for SDE $X_t$ is
\[\begin{split}
\bigl(\gen^{*}_t f\bigr)(x) &=
x a^{\dagger}(x,t) + a(x,t) x^{\dagger} +  \sum_{j=1}^{N} b_j(x,t) b_j^{\dagger}(x,t). \\
\end{split}\]
Hence for low-rank dynamics $\mu_{\lr,t}$, $\gen_{\lr,t}^{*}$ is a
family of generators parametrized by $\opg(t)$, $A(y)$ and $B_j(y)$
($1\le j\le M$), with the following form
\[\begin{split}
 \bigl(\gen_{\lr,t}^{*} f\bigr)(x) &
= x \left( i \opg(t) x + \opu(t) A( \opu^{\dagger}(t) x, t)\right)^{\dagger} + \left( i \opg(t) x + \opu(t) A( \opu^{\dagger}(t) x, t)\right) x^{\dagger} \\
& \qquad + \sum_{j=1}^{M} \opu(t) B_j( \opu^{\dagger}(t) x, t) B_j^{\dagger}(\opu^{\dagger}(t) x, t) \opu^{\dagger}(t). \\
\end{split}\]
\end{lemma}
Note that $\gen_{\lr,t}^{*}$ depends on the current $\opu(t)$, which
is natural. The proof is straightforward by applying \ito{} formula to
$f(x) = xx^{\dagger}$.

We are now ready to apply \eqrefn{eqn::sdlr_minimization} and
\eqref{eqn::sdlr_minimization_equiv} to find the time-evolution of the
optimal low-rank dynamics $\mu_{\lr,t}$.

\medskip 

\begin{theorem}\label{thm::fx=xx}
Assume that $\int f(y)\ \theta_t(\ud y) \equiv \ee_{Y_t\sim \theta_t} \bigl[Y_t Y^{\dagger}_t\bigr]$ is an invertible $r\times r$ matrix for any $t\in \tm$, then 
\begin{enumerate}[(i)]
\item The following choices of $\opu(t)$ and $\theta_t$ give an
  optimal low-rank dynamics 
\begin{equation}
\label{eqn::low_rank_xx}
\begin{cases} 
M  = N\\
A(y, t) =  \opu^{\dagger}(t) a(\opu(t) y, t)  \\
B_j(y, t) = \opu^{\dagger}(t) b_j(\opu(t) y, t) \\
\opg(t) = (-i) \projperp_{\opu(t)} \biggl( \ee_{\theta_t} \bigl[a(\opu(t) y, t) (\opu(t) y)^{\dagger}\bigr] \\ \hspace{10em} + \sum_{j=1}^{N} \ee_{\theta_t}\bigl[b_j(\opu(t) y, t) b_j^{\dagger}(\opu(t) y, t) \bigr] \biggr) \opu(t) \ee_{\theta_t}[yy^{\dagger}]^{-1} \opu^{\dagger}(t) + \hc
\end{cases}
\end{equation}

\item The solution is optimal in the sense that, for any given $\opu(t)$
  and $\theta_t$,
  \begin{equation*}
    \begin{aligned}
      \min_{\wt{\gen}_{\lr,t}}\; d_{\funcsp}\bigl(\wt{\gen}_{\lr,t} \mu_{\lr,t},\
      {\gen}_{t} \mu_{\lr,t} \bigr) & =
      d_{\funcsp}\bigl( {\gen}_{\lr,t}
      \mu_{\lr,t},\ \gen_{t} \mu_{\lr,t} \bigr)\\
      & = \Norm{\sum_{j=1}^{N} \projperp_{\opu(t)} \ee_{\theta_t}
        \bigl[b_j(\opu(t)y,t) b_j^{\dagger}(\opu(t)y,t)\bigr]
        \projperp_{\opu(t)}}_{\hs},
    \end{aligned}
  \end{equation*}
  where the adjoint of ${\gen}_{\lr,t}$ (\ie,
  ${\gen}^{*}_{\lr,t}$) is given by
  \eqrefn{eqn::low_rank_xx}, $\wt{\gen}_{\lr,t}$ is any operator acting on
  probability measures whose corresponding infinitesimal generator
  $\wt{\gen}_{\lr,t}^{*}$ is parametrized by $A$, $B_j$ and $\opg$ as in Lemma
  \ref{lemma::generator}. 

\item Time-evolution equation of $\opu(t)$ and $\theta_t$ is given by the
  following system:
  \begin{equation}\left\{
      \label{eqn::low_rank_xx_v2}
      \begin{split}
        \ud Y &= \opu^{\dagger}(t) a(\opu(t) Y, t)\ \ud t 
        + \sum_{j=1}^{N} \opu^{\dagger}(t) b_j(\opu(t) Y, t)\ \ud W_j\,,\\
        \frac{\ud \opu}{\ud t} &= \projperp_{\opu(t)} \left(\ee\bigl[a(\opu(t) Y, t) Y^{\dagger}\bigr] + \sum_{j=1}^{N} \ee\bigl[b_j(\opu(t) Y, t) b_j^{\dagger}(\opu(t) Y, t)\bigr] \opu(t) \right) \ee\bigl[Y Y^{\dagger}\bigr]^{-1}\,. \\
      \end{split}\right.
  \end{equation}
  If rank $r = n$, then $\projperp_{\opu(t)} = 0$ and $\proj_{\opu(t)} = \id$. Consequently, $\frac{\ud \opu}{\ud t} = 0$ and the infinitesimal generator
  for SDE of $X_{\lr,t}$ is the same as that for $X_t$. That means,
  when full rank is used, the original SDE is recovered; equivalently the Fokker-Planck equation for $\mu_t$ is recovered.
\end{enumerate}

\end{theorem}

\begin{remark}
  Since the time evolution of $\opu(t)$ and $\theta_t$ can be fully
  recovered from solving ODE-SDE coupled system of $Y$ and $\opu$
  (with multiple replica of $Y$), we shall refer
  \eqrefn{eqn::low_rank_xx_v2} as the resulting dynamical low-rank
  approximation as well when no confusion arises.  After all, we will
  not solve equation
  $\partial_t \mu_{\lr,t} = \gen_{\lr,t} \mu_{\lr,t}$ directly;
  instead, we shall use Monte Carlo method, \ie, solving
  \eqrefn{eqn::low_rank_xx_v2} to estimate $\mu_{\lr,t}$ by the
  empirical measure.
\end{remark}

\begin{proof}
  The main idea of the proof is to use stationary conditions with
  respect to $\opg$, $A$ and $B_j$ to derive the low-rank
  dynamics. Recall that $\proj_{\opu} := \opu \opu^{\dagger}$ and
  $\projperp_{\opu} := \id - \opu \opu^{\dagger}$. Fix time $t$, by
  \eqrefn{eqn::sdlr_minimization} and \eqref{eqn::sdlr_minimization_equiv}, we know
\[
\begin{split}
 d_{\funcsp}\bigl(\gen_{\lr,t} \mu_{\lr,t},\ \gen_{t} \mu_{\lr,t} \bigr) 
= \Norm{\ee_{\mu_{\lr,t}} \bigl[\gen_{\lr,t}^{*} f - \gen_{t}^{*} f \bigr] }_{\hs} 
=: \Norm{\opc}_{\hs}, \\
\end{split}
\]
where
\[\begin{split}
 \opc &:=\ee_{\mu_{\lr,t}}\left[x (i \opg(t) x + \opu(t) A(\opu^{\dagger}(t) x, t) - a(x,t))^{\dagger} + \hc \right] \\
 & \qquad + \ee_{\mu_{\lr,t}}\left[ \sum_{j=1}^{M} 
 \opu(t) B_j( \opu^{\dagger}(t) x, t)   
 B_j^{\dagger}(\opu^{\dagger}(t) x, t) 
 \opu^{\dagger}(t)  
 - \sum_{j=1}^{N} b_j(x,t) b_j^{\dagger}(x,t)\right].\\
 \end{split}
\]
To minimize $d_{\funcsp}\bigl(\gen_{\lr,t} \mu_{\lr,t},\ \gen_{t} \mu_{\lr,t} \bigr)$, it is equivalent to minimize $\Norm{\opc}_{\hs}^2 = \Inner{\opc}{\opc}_{\hs}$. The first order stationary conditions of $\Inner{\opc}{\opc}_{\hs}$ with respect to $A$, $B_j$ ($1\le j \le M$) and $\opg(t)$ gives
\begin{equation}
\label{eqn::statn_codn}
\left\{
\begin{split}
0 &= \Inner{\opu^{\dagger}(t) \opc \opu(t)}{\ \ee_{\theta_t}\bigl[(\delta A) y^{\dagger} + y (\delta A)^{\dagger}\bigr]}_{\hs} \\
0 &= \Inner{\opu^{\dagger}(t) \opc \opu(t)}{\ \ee_{\theta_t}\bigl[(\delta B_j) B_j^{\dagger} + B_j (\delta B_j)^{\dagger}\bigr]}_{\hs} \\
0 &= \tr\left( \Comm{\opc}{ \ee_{\mu_{\lr,t}}[ x x^{\dagger}]} \delta \opg \right),
\end{split}\right.
\end{equation}
where $\delta A$ and $\delta B_j$ are perturbations of functions $A$ and $B_j$ respectively and $\delta \opg$ is a Hermitian matrix as perturbation of $\opg$ and $\Comm{\cdot}{\cdot}$ is the commutator. 

Consider the third condition in \eqrefn{eqn::statn_codn}. Let
us complete the basis of $\hbt$ by extending $\text{Ran}(\opu(t))$,
denoted by $\{U_1, U_2, \cdots, U_r, U_{r+1}, \cdots \cdots\}$. Since
$\delta \opg$ is arbitrary among all possible perturbations, consider
the special choice
$\delta \opg = \lambda U_j \Inner{U_k}{\cdot} + \hc$ where $j \le r$
and $k > r$. Note that when both $j, k \le r$ or $j, k > r$, by
\reviewnumtwo{our restriction in \eqrefn{eqn::assmp_g}}, $\Inner{U_j}{\opg(t) U_k} = 0$.  Hence as the perturbation
of $\opg$, $\delta \opg$ should preserve this property.  That means,
we cannot choose $\delta \opg$ with nonzero entries for $j, k \le r$
nor $j, k > r$.  Denote
$\opd \equiv \bigl[\opc, \ee_{\mu_{\lr,t}}[ x x^{\dagger}]\bigr]_{-}$,
which is anti-Hermitian, that is, $\opd^{\dagger} = - \opd$. Plugging the expression of $\delta \opg$ into the third condition in \eqrefn{eqn::statn_codn}, 
we could easily compute that
\[0 = \lambda \Inner{U_k}{\opd U_j} + \lambda^{*} \Inner{U_j}{\opd U_k} = \lambda \Inner{U_k}{\opd U_j} + \lambda^{*} \Inner{U_k}{\opd^{\dagger} U_j}^{*} = \lambda \Inner{U_k}{\opd U_{j}}  - c.c. \]
Hence $\text{Im}\left(\lambda \Inner{U_k}{\opd U_j}\right) = 0$. 
Since $\lambda\in \Complex$ is arbitrary, $\Inner{U_k}{ \opd U_j} = 0$ for all $j\le r$ and $k > r$.  
That is to say, $\projperp_{\opu(t)} \opd \proj_{\opu(t)} = 0$. By plugging the expression of $\opd\equiv \bigl[\opc, \ee_{\mu_{\lr,t}}[ x x^{\dagger}]\bigr]_{-}$, 
we could compute that 
\[ \begin{split} 
0 
&= \projperp_{\opu(t)} \opc  \opu(t) \ee_{\theta_t} \bigl[yy^{\dagger}\bigr] \opu^{\dagger}(t)\\
&= \ee_{\theta_t} \Biggl[ i \projperp_{\opu(t)} \opg(t) \opu(t) y y^{\dagger} 
- \projperp_{\opu(t)} a(\opu(t)y,t) y^{\dagger}\\
& \hspace{4em} -\projperp_{\opu(t)}\sum_{j=1}^{N} b_j(\opu(t)y,t) b_j^{\dagger}(\opu(t)y,t) \opu(t) \Biggr] \ee_{\theta_t}\bigl[yy^{\dagger}\bigr] \opu^{\dagger}(t) \\
&= i \projperp_{\opu(t)} \opg(t) \opu(t) \ee_{\theta_t} \bigl[y y^{\dagger}\bigr]^2 \opu^{\dagger}(t) 
- \projperp_{\opu(t)} \ee_{\theta_t} \bigl[a(\opu(t)y,t) y^{\dagger}\bigr] \ee_{\theta_t}\bigl[y y^{\dagger}\bigr]\opu^{\dagger}(t) \\
&- \projperp_{\opu(t)} \sum_{j=1}^{N} \ee_{\theta_t}\bigl[ b_j(\opu(t)y,t) b_j^{\dagger}(\opu(t)y,t) \bigr] \opu(t) \ee_{\theta_t}[ y y^{\dagger}] \opu^{\dagger}(t).
 \end{split}\]
Multiply both sides by $\opu(t)\ee_{\theta_t}\bigl[yy^{\dagger}\bigr]^{-2}$ on the right, the last equation yields
\[i \projperp_{\opu(t)} \opg(t) \opu(t)=\projperp_{\opu(t)} \left( \ee_{\theta_t} \bigl[ a(\opu(t)y,t)  y^{\dagger}\bigr] + \sum_{j=1}^{N} \ee_{\theta_t} \bigl[b_j(\opu(t)y,t) b_j^{\dagger}(\opu(t)y,t) \bigr] \opu(t)\right) \ee_{\theta_t}[y y^{\dagger}]^{-1}. \]
Then multiply $\opu^{\dagger}(t)$ on the right, and then divide both side by $i$, one could obtain
\[\begin{split}
 \projperp_{\opu(t)} \opg(t) \proj_{\opu(t)} 
 =&  (-i) \projperp_{\opu(t)} \left( \ee_{\theta_t} \bigl[a(\opu(t)y,t) (\opu(t) y)^{\dagger}\bigr] + \sum_{j=1}^{N} \ee_{\theta_t} \bigl[b_j(\opu(t)y,t) b_j^{\dagger}(\opu(t)y,t) \bigr] \right) \opu(t) \ee_{\theta_t}[y y^{\dagger}]^{-1} \opu^{\dagger}(t).\\
 \end{split} \]
Thus we have already obtained the expression of $\opg(t) \equiv \projperp_{\opu(t)} \opg(t) \proj_{\opu(t)} + \hc$ (cf. \eqrefn{eqn::low_rank_xx}).
{ As a remark, up to here, we have not yet used any information nor assumption about $A$ and $B_j$. }

Then we plug the expression of $\opg(t)$ into $\opc$,
\[\begin{split}
\opc &= \proj_{\opu(t)} \ee_{\theta_t} \bigl[ \bigl(\opu(t) A(y,t) - a(\opu(t)y,t)\bigr)y^{\dagger} \opu^{\dagger}\bigr] + \hc\\
& \qquad + \sum_{j=1}^{M} \opu(t) \ee_{\theta_t} \bigl[B_j(y,t) B_j^{\dagger}(y,t)\bigr] \opu^{\dagger}(t)  \\
& \qquad - \sum_{j=1}^{N} \left( \proj_{\opu(t)} \ee_{\theta_t} \bigl[b_j(\opu(t)y,t) b_j^{\dagger}(\opu(t)y,t)\bigr] \proj_{\opu(t)} 
 + \projperp_{\opu(t)} \ee_{\theta_t} \bigl[b_j(\opu(t)y,t) b_j^{\dagger}(\opu(t)y,t)\bigr] \projperp_{\opu(t)}\right). \\
\end{split}
\]
One could observe that $\opc = \proj_{\opu(t)} \opc \proj_{\opu(t)} + \projperp_{\opu(t)} \opc \projperp_{\opu(t)}$ with
\[\left\{
\begin{split}
\proj_{\opu(t)} \opc \proj_{\opu(t)} &= \proj_{\opu(t)} \ee_{\theta_t} \bigl[ \bigl(\opu(t) A(y,t) - a(\opu(t)y,t)\bigr)y^{\dagger} \opu^{\dagger}\bigr] + \hc + \sum_{j=1}^{M} \opu(t) \ee_{\theta_t} \bigl[B_j(y,t) B_j^{\dagger}(y,t)\bigr] \opu^{\dagger}(t) \\
&  - \sum_{j=1}^{N}\proj_{\opu(t)} \ee_{\theta_t} \bigl[b_j(\opu(t)y,t) b_j^{\dagger}(\opu(t)y,t)\bigr] \proj_{\opu(t)}, \\
\projperp_{\opu(t)} \opc \projperp_{\opu(t)} &= - \sum_{j=1}^{N} \projperp_{\opu(t)} \ee_{\theta_t} \bigl[b_j(\opu(t)y,t) b_j^{\dagger}(\opu(t)y,t)\bigr] \projperp_{\opu(t)}.
\end{split}\right.\]
Notice that $\projperp_{\opu(t)} \opc \projperp_{\opu(t)}$ does not depend on any parameter we choose, that is, it is independent of $A$, $B_j$ and $\opg(t)$. 
Recall that we would like to minimize 
\[\Inner{\opc}{\opc}_{\hs} 
= \Inner{\proj_{\opu(t)} \opc \proj_{\opu(t)}}{\proj_{\opu(t)} \opc \proj_{\opu(t)}}_{\hs} 
+ \inner{\projperp_{\opu(t)} \opc \projperp_{\opu(t)}}{\projperp_{\opu(t)} \opc \projperp_{\opu(t)}}_{\hs}.\] The second term is non-negative and we cannot minimize it further. As for the first term, by choosing 
\[\left\{
\begin{split}
M &= N, \\
A(y,t) &=  \opu^{\dagger}(t) a(\opu(t) y, t), \\
B_j(y,t) &= \opu^{\dagger}(t) b_j(\opu(t)y,t), \\
\end{split}\right.
\]
one could easily verify that $\proj_{\opu(t)} \opc \proj_{\opu(t)} = 0$. With such choice, 
\[\Norm{\opc}_{\hs} = \Norm{\sum_{j=1}^{N} \projperp_{\opu(t)} \ee_{\theta_t} \bigl[b_j(\opu(t)y,t) b_j^{\dagger}(\opu(t)y,t)\bigr] \projperp_{\opu(t)}}_{\hs}.\] 
Hence the above choice must be optimal (although it does not imply
uniqueness). 
However, the choice of $\opg$ is indeed unique under \reviewnumtwo{our restriction}.
One could straightforwardly verify that this solution
satisfies the first two equations in the first order stationary
condition (\ie, \eqrefn{eqn::statn_codn}); that is to say, the above
choice yields $\opu^{\dagger}(t) \opc \opu(t) = 0$. Though we don't
have to use first two parts in \eqrefn{eqn::statn_codn} to derive
low-rank dynamics, it is still nice to observe the consistency.

Thus we have proved the first and second part of this theorem. The third part follows easily from the first part.
\end{proof}

Let us come back to the two restrictions we made at the beginning of
this section.
\begin{itemize}
\item The reason to impose the condition
  $\opg(t) = \projperp_{\opu(t)} \opg(t) \proj_{\opu(t)} + \hc$ is to
  remove the redundant degree of freedom: It is easy to observe that $\projperp_{\opu(t)} \opg(t) \projperp_{\opu(t)}$ is redundant and does not play any role in $\dot{\opu}(t)$. 
Thus, we might as well let $\projperp_{\opu(t)} \opg(t) \projperp_{\opu(t)} \equiv 0$.
Further if we assume $\proj_{\opu(t)} \opg(t)\proj_{\opu(t)} \equiv 0$, then it directly implies that $\opu^{\dagger}(t) \dot{\opu}(t) = 0$, which is used as orthogonal constraint in \cite{Lubich07}. Conversely, if $\opu^{\dagger}(t) \dot{\opu}(t) = 0$, then $\proj_{\opu(t)} \opg(t) \proj_{\opu(t)} = 0$, which shows that such constraint is similar to the constraint $\opu^{\dagger}(t) \dot{\opu}(t) = 0$ for matrix ODEs. 

Another reason comes from the above proof. With such constraint, the first order stationary condition with respect to $\delta \opg$ yields a unique expression for $\opg(t)$. This indicates that redundancy has been removed via the above constraint.

\item The main reason for imposing constraint in tangent space of $\manitheta$ is that Fokker-Planck type generator automatically helps to preserve the positivity of measure. In fact,  the whole tangent space is rather big and too opaque to handle,
  since it might involve generators for other stochastic processes,
  \eg, jump processes.  The mixture of jump process and diffusion makes
  it more challenging to derive a simple low-rank dynamics; which
  could be an interesting future research direction.
\end{itemize}

\medskip

\begin{remark}
  It is a good place to compare our approach with the dynamical
  orthogonal (DO) method \cite{Sapsis09}. Using the current notations,
  the ansatz in DO method is taken to be 
\begin{equation}
X_{\lr,t}(\omega) = \bar{X}_t + \opu(t) Y_t(\omega),
\label{eqn::do_assmp}
\end{equation}
where $\bar{X}_t$ is deterministic and for all $t$, $\ee[Y_t(\omega)]= 0$, $\opu^{\dagger}(t) \dot{\opu}(t) = 0$, $\opu^{\dagger}(t) \opu(t) = \id$. By re-deriving the low-rank dynamics following the proof in that paper, one could obtain that
\begin{equation}
\label{eqn::do}
\begin{split}
  \frac{\ud}{\ud t}\bar{X} &= \ee\bigl[a(X_{\lr,t},t)\bigr] \\
  \dot{\opu} &= \projperp_{\opu(t)} \ee\bigl[a(X_{\lr,t},t) Y^{\dagger}\bigr] \ee\bigl[Y Y^{\dagger}\bigr]^{-1}\\
  \ud Y &= \opu^{\dagger}(t) \left(a(X_{\lr,t},t) - \ee\bigl[a(X_{\lr,t},t)\bigr]\right) \ud t + \sum_{j=1}^{N} \opu^{\dagger}(t) b_j(X_{\lr,t},t)\ud W_j.\\
\end{split}
\end{equation}
Compared with \eqrefn{eqn::low_rank_xx_v2}, the expressions of
time-evolution of $Y$ and $\opu(t)$ are almost the same except: (1) the extra term
$\opu^{\dagger}(t) \ee[a(X_{\lr,t},t)]$ in DO method, due to the
zero-mean constraint in $Y$; and (2) the non-trivial difference that
in \eqrefn{eqn::low_rank_xx_v2}, we have a term
\[\sum_{j=1}^{N} \ee[b_j(X_{\lr,t},t) b_j^{\dagger}(X_{\lr,t},t)]
\opu(t),\]
which could be understood as the \ito{} correction term due to the
second moment. Some numerical experiments comparing \thismethodshort{}
method and DO method will be presented in \secref{sec::numerics} for
high-dimensional geometric Brownian motion and stochastic Burgers'
equation.
\end{remark}


\section{Error analysis of the stochastic dynamical low-rank approximation}
\label{sec::error}

In this section, we provide some error analysis for the low-rank dynamics that
we derive in \secref{sec::example}.
\thmref{thm::recover} indicates that the low-rank dynamics
in \eqrefn{eqn::low_rank_xx} (or \eqrefn{eqn::low_rank_xx_v2}) is optimal
under our ansatz, in the sense that if $X_t$ is itself a SDE, whose range is
supported on a rank $r$ subspace, then $X_{\lr,t} = X_t$ by choosing rank-$r$ low-rank dynamics \addchange{with some additional assumptions.}
In other words, the
low-rank dynamics we derive is consistent.  
We also prove an inequality to bound
error propagation in \thmref{thm::bound}  based on \gwieq.

\subsection{Consistency: recovering low-rank dynamics}
\label{subsec::recover}

\begin{lemma}
	\label{lemma::vanish_diffusion}
	Suppose $\mu_t$ is the measure induced by $X_t$ which solves a SDE of
	the form as in \eqrefn{eqn::sde} 
	\begin{equation*}
	\ud X_t = a(X_t, t) \ud t + \sum_{j=1}^N b_j(X_t, t) \ud W_j,
	\end{equation*}
	with initial condition $X_0 \sim \mu_0$.  Assume that $\mu_t$ is
	supported on a $r$-dimensional linear subspace, whose basis forms a
	linear operator $\opv(t):\Complex^r\rightarrow \hbt$ 
	with
	$\opv(t)^{\dagger} \opv(t) = \id_{r\times r}$. 
	Then, for $x\in \ran(\opv(t))$,
	the coefficients of the SDE satisfy
	\begin{equation*}
	a(x,t) = \left(\frac{\ud}{\ud t} \proj_{\opv(t)} \right) x + \proj_{\opv(t)} a(x,t), \qquad \text{and} \qquad b_j(x,t) = \proj_{\opv(t)} b_j(x,t), \qquad \forall j,  
	\end{equation*}
	and hence
	\[\sum_{j=1}^{N} \projperp_{\opv(t)}  \ee_{\mu_t} \bigl[ b_j(x,t) b_j^{\dagger}(x,t) \bigr] \projperp_{\opv(t)} = 0.\]
\end{lemma}

\begin{proof}
	Since the range of $X_t$ is $\ran(\opv(t))$ by assumption, $\opv(t) \opv(t)^{\dagger} X_t = X_t$. Take derivative for both sides (note that $\opv(t)$ is deterministic)
	\begin{equation}
	\label{eqn::recover_lowrank}
	\begin{split}
	\ud X_t &= \ud \bigl(\proj_{\opv(t)} X_t\bigr) 
	= \left(\frac{\ud}{\ud t}\proj_{\opv(t)}\right) X_t \ud t + \proj_{\opv(t)} \ud X_t \\
	&=  \left(\left(\frac{\ud}{\ud t}\proj_{\opv(t)}\right) X_t + \proj_{\opv(t)} a(X_t, t)\right) \ud t + \sum_{j=1}^{N} \proj_{\opv(t)} b_j(X_t, t) \ud W_j. \\
	\end{split}\end{equation}
	Matching it with the SDE of $X_t$, we obtain 
	$a(X_t, t) = \left(\frac{\ud}{\ud t}\proj_{\opv(t)}\right) X_t + \proj_{\opv(t)}
	a(X_t, t)$
	and $b_j(X_t, t) = \proj_{\opv(t)} b_j(X_t, t)$. The conclusion of the
	Lemma follows.
\end{proof}

\begin{theorem}
	\label{thm::recover}
	Besides the same assumptions of Lemma~\ref{lemma::vanish_diffusion},
	we further assume that 
	$\frac{\ud}{\ud t}\opv(t) = i \opf(t) \opv(t)$ with some Hermitian matrix
	$\opf(t)$ which satisfies 
	\begin{equation*}
	\opf(t) = \projperp_{\opv(t)} \opf(t) \proj_{\opv(t)} + \hc
	\end{equation*}
	Then if $\mu_0 = \mu_{\lr,0}$ and $\opv(0) = \opu(0)$, we have
	\begin{equation*}
	\mu_t = \mu_{\lr,t},\quad \text{and} \quad \opv(t) = \opu(t)
	\end{equation*}
	for all $t\in\tm$ as long as the low-rank dynamics in
	\eqrefn{eqn::low_rank_xx} exists (in particular,
	$\ee_{\theta_t}[yy^{\dagger}]$ remains invertible). 
\end{theorem}

\begin{proof}
	Fix time $t$ and assume $X_t = X_{\lr,t}$ in distribution and
	$\opv(t) = \opu(t)$. 
	Next, we shall show that $X_{\lr,t}$ and $ X_t$ satisfy the same SDE locally 
	and $\frac{\ud}{\ud t} \opv(t) = \frac{\ud }{\ud t} \opu(t)$. 
	
	By \thmref{thm::fx=xx}, we could
	straightforwardly compute that
	\[\begin{split}
	\ud X_{\lr,t} &= \ud \bigl( \opu(t) Y_t \bigr) = \left( \frac{\ud \opu(t)}{\ud t}  Y_t \right) \ud t + \opu(t) \ud Y_t  \\
	&= \frac{\ud \opu(t)}{\ud t}  \opu^{\dagger}(t) X_{\lr,t} \ud t + \opu(t) \left(\opu^{\dagger}(t) a(X_{\lr,t},t)\ \ud t + \sum_{j=1}^{N} \opu^{\dagger}(t) b_j(X_{\lr,t}, t) \ud W_j \right) \\
	&= \left(\frac{\ud \opu(t)}{\ud t}  \opu^{\dagger}(t) X_{\lr,t} + \opu(t) \opu^{\dagger}(t) a(X_{\lr,t},t) \right) \ud t + \sum_{j=1}^{N} \opu(t) \opu^{\dagger}(t) b_j(X_{\lr,t}, t) \ud W_j \\
	&= \left( \left(\frac{\ud}{\ud t}\proj_{\opu(t)}\right) X_{\lr,t} + \proj_{\opu(t)} a(X_{\lr,t}, t) \right) \ud t + \sum_{j=1}^{N} \proj_{\opu(t)} b_j(X_{\lr,t}, t) \ud W_j .\\
	\end{split}
	\]
	In the last step, we have used the fact that $\opu(t) \frac{\ud \opu^{\dagger}(t)}{\ud t} X_{\lr,t} = 0$ due to \reviewnumtwo{the restriction} on Hermitian matrix $\opg(t)$ (cf. \eqrefn{eqn::assmp_g}). 
	Comparing with \eqrefn{eqn::recover_lowrank}, we obtain that
	the SDEs of $X_t$ and $X_{\lr, t}$ coincide since $\opv(t) = \opu(t)$, as long as $\frac{\ud}{\ud t}\proj_{\opv(t)} = \frac{\ud}{\ud t} \proj_{\opu(t)}$. That means, we still need to verify $\frac{\ud }{\ud t} \opv(t) = \frac{\ud }{\ud t} \opu(t)$.
	
	Moreover, by \eqrefn{eqn::low_rank_xx_v2},
	\[\begin{split}
	\frac{\ud}{\ud t} \opu(t) &= \projperp_{\opu(t)} \left(\ee\bigl[a(X_{\lr},t) Y^{\dagger}\bigr] + \sum_{j=1}^{N} \ee\bigl[b_j(X_{\lr},t) b_j^{\dagger}(X_{\lr},t)\bigr] \opu(t) \right) \ee[Y Y^{\dagger}]^{-1}\\
	&= \projperp_{\opu(t)} \ee\bigl[a(X_{\lr},t) Y^{\dagger}\bigr] \ee[Y Y^{\dagger}]^{-1}, \qquad  \bigl(\text{by Lemma }\ref{lemma::vanish_diffusion} \text{ and } \opu(t) = \opv(t) \bigr) \\
	&= \projperp_{\opu(t)} \ee\left[ \left(\frac{\ud}{\ud t} \proj_{\opv(t)}\right) X_{\lr,t}  Y^{\dagger} + \proj_{\opv(t)} a(X_{\lr,t},t) Y^{\dagger} \right] \ee[Y Y^{\dagger}]^{-1}, \qquad (\text{by \lemref{lemma::vanish_diffusion}} ) \\
	&= \projperp_{\opu(t)}\left( \frac{\ud}{\ud t} \proj_{\opv(t)}\right) \opu(t) = i \opf(t) \opv(t)  = \frac{\ud}{\ud t} \opv(t).
	\end{split}\]
	Therefore $\frac{\ud}{\ud t} \opv(t) = \frac{\ud}{\ud t} \opu(t)$ and the
	conclusion holds.
\end{proof}

\subsection{Error bound}

We shall quantify error defined by 
\[E(t) := \Norm{ \ee_{\mu_t} \bigl[ xx^{\dagger} \bigr] - \ee_{\mu_{\lr,t}}
	\bigl[ xx^{\dagger}\bigr] }_{\hs}, \]
which measures the difference of second moment for two measures
$\mu_t$ and $\mu_{\lr,t}$ in Hilbert-Schmidt norm. Recall that we have
chosen the single test function $f(x) = x x^{\dagger}$.

\smallskip 
\begin{theorem}
	\label{thm::bound}
	Assume that:
	\begin{enumerate}
		\item Throughout the time-evolution of low-rank dynamics for $t\in \tm$,
		\[d_{\funcsp}\bigl( {\gen}_{\lr,t}
		\mu_{\lr,t},\ \gen_{t} \mu_{\lr,t} \bigr) \equiv \Norm{\sum_{j=1}^{N} \projperp_{\opu(t)} \ee_{\theta_t} \bigl[b_j(\opu(t)y,t) b_j^{\dagger}(\opu(t)y,t)\bigr] \projperp_{\opu(t)}}_{\hs} \le \eps^2. \]
		Recall that $\theta_t$ is the pullback of the restriction of $\mu_{\lr,t}$ to
		$\ran(\opu(t))$.  
		\item There exists a function $\gamma(t)$ such that, for $f(x) = xx^{\dagger}$, 
		\[\Norm{ \ee_{\mu_t} \bigl[ \gen_{t}^* f \bigr] -
			\ee_{\mu_{\lr,t}}\bigl[\gen_t^{*} f \bigr] }_{\hs} \le \gamma(t)
		E(t) \equiv \gamma(t) \Norm{ \ee_{\mu_t} \bigl[ f \bigr] -
			\ee_{\mu_{\lr,t}} \bigl[ f\bigr] }_{\hs}. \]
	\end{enumerate}
	Then, the error satisfies the integral inequality
	\[ E(t) \le E(0) + \eps^2 t + \int_{0}^{t} \gamma(s) E(s) \ud s,\]
	and thus by \gwieq, 
	\[E(t) \le \bigl(E(0) + \eps^2 t\bigr) \exp\left(\int_{0}^{t} \gamma(s)\ud s \right).\]
\end{theorem}

\begin{proof}
	The proof follows from standard error analysis for time evolution equations:
	\[\begin{split}
	E(t) & \equiv \Norm{\ee_{\mu_t} \bigl[ f \bigr] - \ee_{\mu_{\lr,t}}\bigl[ f\bigr]}_{\hs} \\
	& \le E(0) + \Norm{\int_{\hbt} f(x) \bigl(\mu_t - \mu_0\bigr)(\rd x) - f(x) \bigl(\mu_{\lr,t} - \mu_{\lr,0}\bigr)(\rd x)   }_{\hs}\\
	& = E(0) + \Norm{\int_{\hbt} f(x) \left(\int_{0}^{t}\gen_s \mu_s\ud s \right)(\rd x) - f(x) \left( \int_{0}^{t} \gen_{\lr,s} \mu_{\lr,s}\ud s \right)(\rd x)   }_{\hs}\\
	&= E(0) + \Norm{\int_{0}^{t} \left( \int_{\hbt} f(x) \left(\gen_s \mu_s \right)(\rd x) - f(x) \left(\gen_{\lr,s} \mu_{\lr,s} \right)(\rd x)\right)\   \rd s }_{\hs}  \\
	&= E(0) + \Norm{\int_{0}^{t} \left( \int_{\hbt} \left(\gen_{s}^{*} f\right) (x) \mu_s(\rd x) - \left(\gen_{\lr,s}^{*}  f\right)(x)  \mu_{\lr,s} (\rd x)\right)\ \rd s   }_{\hs} \\
	&\le E(0) + \int_{0}^{t}  \Norm{\ee_{\mu_s} \bigl[\gen_s^{*} f \bigr] - \ee_{\mu_{\lr,s}}\bigl[\gen_s^{*} f\bigr] + \ee_{\mu_{\lr,s}}\bigl[\gen_s^{*} f\bigr] - \ee_{\mu_{\lr,s}}\bigl[\gen_{\lr,s}^{*} f\bigr]  }_{\hs}\ \rd s \\
	&\le E(0) + \int_{0}^{t}  \Norm{\ee_{\mu_s} \bigl[\gen_s^{*} f \bigr] - \ee_{\mu_{\lr,s}}\bigl[\gen_s^{*} f\bigr]}_{\hs}\ \rd s + \int_{0}^{t} \Norm{ \ee_{\mu_{\lr,s}}\bigl[\gen_s^{*} f\bigr] - \ee_{\mu_{\lr,s}}\bigl[\gen_{\lr,s}^{*} f\bigr]  }_{\hs}\ \rd s \\
	& = E(0) + \int_{0}^{t}  \Norm{\ee_{\mu_s} \bigl[\gen_s^{*} f \bigr] - \ee_{\mu_{\lr,s}}\bigl[\gen_s^{*} f\bigr]}_{\hs}\ \rd s 
	+ \int_{0}^{t} d_{\funcsp}\left(\gen_s \mu_{\lr,s}, \gen_{\lr,s} \mu_{\lr,s} \right)\ \rd s \\ 
	&\le E(0) + \int_{0}^{t} \gamma(s) E(s)\ \ud s + \int_{0}^{t} \eps^2\ \rd s. \\
	\end{split}\]
	Thus we have proved the integral inequality; the rest of the conclusion follows from \gwieq. 
\end{proof}

\begin{remark}
	The estimate above holds for arbitrary $\eps$, however in practice,
	the above bound is most useful when $\eps$ is a small number. This
	corresponds to that the diffusivity function
	$\projperp_{\opu(t)} b_j(\opu(t)y,t) =\mathcal{O}(\eps)$.
	For instance, SDE system with small noise falls
	into this type. 
	When $\eps$ is
	large, of course, this indicates that low-rank approximation
	fails. One might need to use higher rank  to get an
	accurate approximation of the dynamics, or even use full rank (\ie, solving the original dynamics). 
\end{remark}

From the proof, the assumption (2) in \thmref{thm::bound}
is natural, in order to use \gwieq.  To better illustrate the
assumption, \reviewnumtwo{we provide here a concrete example to give explicit
form of $\gamma(t)$.}

\medskip

\reviewnumtwo{
\begin{example}[Choice of $\gamma(t)$ for linear drift and diffusion functions]
	\normalfont
	\
		If SDE in \eqrefn{eqn::sde} has linear drift and diffusion functions, that is,
		\[a(x,t) = \vect{\Lambda}(t) x, \qquad b_j(x,t) = \vect{\Theta}_j(t) x,\]
		where bold Greek letters $\vect{\Lambda}$ and $\vect{\Theta}_j$ 
		are time-dependent matrices on $\hbt$. 
		
		By Lemma \ref{lemma::generator}, one could straightforwardly find that 
		\[\gamma(t) = 2\Norm{\vect{\Lambda}(t) }_{2} + \sum_{j=1}^{N} \Norm{\vect{\Theta}_j(t)}_2^2\]
		satisfies the assumption.
\end{example}
Another example of $\gamma(t)$ for unraveling scheme will be given in the next section below.
}

Consider that if the rank $r = \dimn$, then
$\projperp_{\opu(t)} = 0$ for all $t$, so that $\eps = 0$. If we further
let $E(0) = 0$, then $E(t) = 0$ for all $t$.  This result is
consistent with intuition, since in this case the low-rank dynamic is
exactly the original SDE.  Though the rank $r$ does not appear
explicitly in the error estimate, it is implicitly hidden inside
$\eps$. Loosely speaking, the larger the rank $r$, the smaller the
$\eps$. 
\reviewnumtwo{
	Even though the relation between $\eps$ and $r$ is not analytically given, \thmref{thm::bound} can still be useful in practical simulation and in designing adaptive scheme.   
	Numerically what we need is to set up error tolerance $\eps$, and then compute $\Norm{\sum_{j=1}^{N} \projperp_{\opu(t)} \ee_{\theta_t} \bigl[b_j(\opu(t)y,t) b_j^{\dagger}(\opu(t)y,t)\bigr] \projperp_{\opu(t)}}_{\hs}$ on-the-fly in
	the dynamics; if this quantity is close to $\eps^2$, then it indicates that the rank chosen is not large enough anymore and we should adaptively increase the rank $r$ in order to control the error.}
This idea has been used in
\cite{LeBris13} in numerically solving \lb{} equation by the
deterministic low-rank approximation.


\section{Connections to unraveling of Lindblad equations}
\label{sec::discuss::unravel}
In this section, we shall discuss the relationship between unraveling
of \lb{} equation and dynamical low-rank approximation method; in
particular, we establish a commuting diagram for the action of
unraveling and the action of dynamical low-rank approximation under
certain conditions.

\subsection{\lb{} equation and stochastic unraveling}

\lb{} equation, one of the most popular quantum master equations for
open quantum systems, has the following form \cite{Lindblad76}
\begin{equation}
\label{eqn::lb}
\dot{\oprho} = \lbop(\oprho) =  -i \Comm{\ophami}{\oprho} + 
\sum_{k} \left( \oplb_{k} \oprho \oplb_{k}^{\dagger} - \frac{1}{2} \Anticomm{\oplb_{k}^{\dagger} \oplb_{k}}{ \oprho }  \right),
\end{equation}
where Hermitian operator $\ophami$ is Hamiltonian, 
$\oplb_k$ are Lindblad operators, \addchange{
$\Comm{\cdot}{\cdot}$ is the commutator} and
$\Anticomm{\cdot}{\cdot}$ is the anti-commutator. The second term on
the right hand side models the interaction of the system with the
environment.

\begin{figure}[h!]
\begin{center}
%
%
%
%
%
%
%
%
\includegraphics[width=\textwidth]{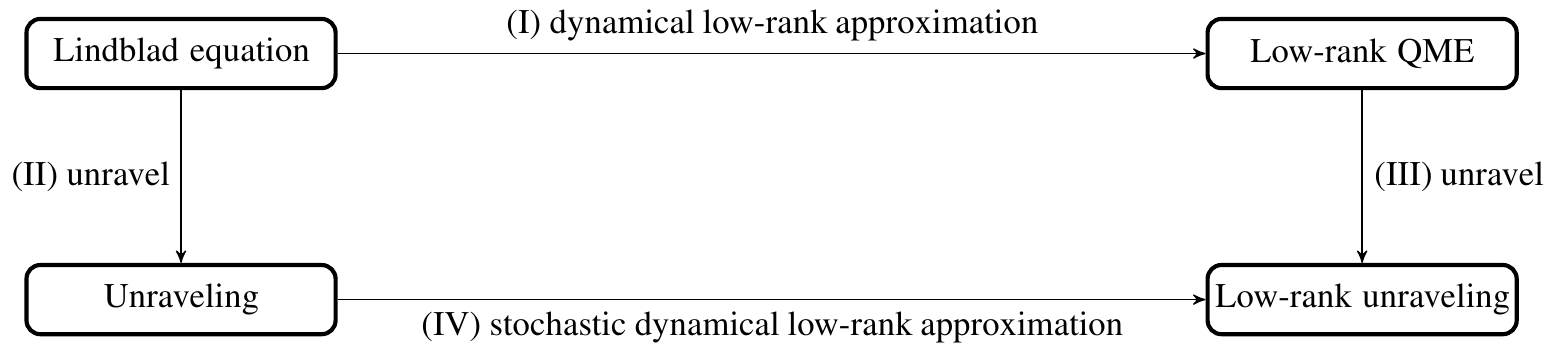}
\end{center}
\caption{Diagram for unraveling and dynamical low-rank approximation
  for \lb{} equations.}
\label{fig::diagram}
\end{figure}

The \lb{} equation is often challenging to solve numerically due to
its high-dimensionality. There are two major dimension-reduction
approaches from the literature, summarized in the
\figref{fig::diagram}: (stochastic) unraveling method (II) and
(deterministic) dynamical low-rank approximation method (\rom{1}).

{\bf Step (\rom{1}): Dynamical low-rank approximation method}. This
method for \lb{} equation has been studied in \cite{LeBris13}: The
ansatz $ \oprho_{\lr}(t) = \LU (t) \LSIG (t) \LU^{\dagger}(t)$ was
used, where $\LU(t)$ satisfies the orthonormality constraint
$\LU^{\dagger}(t) \LU(t) = \id$; and $\LSIG(t)$ is a $r\times r$
strictly positive matrix with trace one. Tilde is used to distinguish
the $\opu(t)$ and $\opsigma(t)$ for deterministic dynamical low-rank
approximation and those in the \thismethodshort{} method.  Dynamical
low-rank approximation method would lead to a coupled ODE system for
$\LU(t)$ and $\LSIG(t)$, which approximates the \lb{} equation of
$\oprho(t)$. We shall revisit the result from \cite{LeBris13} below in
\thmref{thm::lebris}, while dropping the trace-preserving constraint to $\LSIG$. We will discuss the trace-preserving constraint further in \secref{sub:trace}.

\medskip 
  
  \begin{theorem}[Adapted from \cite{LeBris13} with modification]
\label{thm::lebris}
Consider the subspace
\[ \wt{\mani}_r := \left\{ \LU \LSIG \LU^{\dagger}: \LU^{\dagger} \LU = \id_{r\times r},\ \LSIG > 0 \right\}, \] 
to approximate the manifold of positive matrices. 
Then by dynamical low-rank approximation, \ie, by solving 
\[\frac{\ud}{\ud t}\oprho_{\lr}(t) = \argmin_{v \in \tang_{\oprho_{\lr}(t)} \mani_r} \Norm{\lbop\bigl(\oprho_{\lr}(t)\bigr) - v }_{\hs},\]
with further restriction on the tangent space that $\frac{\ud}{\ud t}\LU(t) = i \opg(t)\LU(t)$, where Hermitian matrix 
$\opg(t) = \projperp_{\LU(t)} \opg(t) \proj_{\LU(t)} + \hc$, we have the unique low-rank dynamics, 
\begin{equation}
\label{eqn::lebris}
\left\{
\begin{split}
\frac{\ud}{\ud t}\LSIG(t) &= \LU^{\dagger}(t) \lbop\bigl(\oprho_{\lr}(t)\bigr) \LU(t)\\
\frac{\ud}{\ud t} \LU(t) &= \projperp_{\LU(t)} \lbop\bigl(\oprho_{\lr}(t) \bigr) \LU(t) \LSIG(t)^{-1}. \\
\end{split}\right.
\end{equation}
\end{theorem}

The proof is almost the same as the derivation in \cite{LeBris13}. It
is provided in \ref{app::proof_lebris} for readers' convenience. It
should be noticed that if the condition
$\opg(t) = \projperp_{\opu(t)} \opg(t) \proj_{\opu(t)} + \hc$ is not
imposed, the low-rank dynamics is not unique as additional degrees of
freedom exist, called $A$ in \cite{LeBris13}. The above dynamics could
be viewed as a special case by choosing $A= 0$ in that reference. Just
to be clear, note that this $A$ is not the drift function in this
paper; it is a notation used in \cite{LeBris13}.  However, the major
difference comes from the trace-preserving condition in
\cite{LeBris13}, which takes the subspace as
$\wt{\mani}_r := \bigl\{ \LU \LSIG \LU^{\dagger}: \LU^{\dagger} \LU =
  \id_{r\times r},\ \LSIG > 0,\ \tr(\LSIG) = 1 \bigr\}$.
As a consequence, $\lambda \id_{r\times r}$ term in equation (5) in
\cite{LeBris13} does not appear in our expression.

\textbf{Step (\rom{2}): Unraveling}. Recall that unraveling means a stochastic
  wave-equation that recovers the evolution of density matrices
  $\oprho(t)$ in expectation. More specifically, it looks for a
  stochastic process $X_t(\omega)$ on the Hilbert space $\hbt$, such
  that the expectation
  $\oprho(t):=\ee\bigl[X_t(\omega) X_t^{\dagger}(\omega)\bigr]$ solves
  the \lb{} equation.  Multiple choices exist for the unraveling
  stochastic process: quantum state diffusion (QSD) \cite{Gisin92},
  linear quantum state diffusion (LQSD) \cite{Brun00} and quantum jump
  process \cite{Dalibard92}. We restrict the stochastic unraveling to
  SDE type in the discussion below, that is, we shall only consider
  unraveling of the form as in \eqrefn{eqn::sde}. Still various
  choices exist, while the coefficients satisfy the relation stated in
  the following Lemma.  \smallskip
  \begin{lemma}
    \label{lemma::lb_unravel}
    If $X_t(\omega)$ in \eqrefn{eqn::sde} is a stochastic unraveling of a \lb{} equation $\dot{\oprho} = \lbop(\oprho)$ if and only if
    \begin{equation}
      \label{eqn::lb_equiv_cond}
      a(x,t) x^{\dagger} + x a^{\dagger}(x,t) + \sum_{j=1}^{N} b_j(x,t) b_j^{\dagger}(x,t) = \lbop\bigl(xx^{\dagger}\bigr),\qquad \forall x\in \hbt,\ \forall t\in \tm.
    \end{equation}
  \end{lemma}

  Among various options of unraveling, two most popular choices are:
  \begin{itemize}
  \item  Linear quantum state diffusion (LQSD) refers to the choice
    \[\left\{\begin{split}
        a(x) &= \left(-i \ophami - \frac{1}{2} \sum_{k} \oplb_k^{\dagger} \oplb_k \right) x\\
        b_{k,1}(x) &= \frac{1}{\sqrt{2}} \oplb_k x \\
        b_{k,2}(x) &= \frac{i}{\sqrt{2}} \oplb_k x.\\
      \end{split}\right.\]
    So that the autonomous SDE is given by 
    \[\ud X_t = a(x) \ud t +\sum_{k} \left( b_{k,1}(x)\ud W_{k,1} + b_{k,2}(x)  \ud W_{k,2}\right), \]
    where $W_{k,1}$ and $W_{k,2}$ are independent standard
    (real-valued) Brownian motions. If we have $N$ \lb{} operators,
    then we have $2N$ (real-valued) Brownian motions in SDE. If one combines
    $W_{k,1}$ and $W_{k,2}$ together to form a complex-valued Brownian
    motion $W_{k,1} + i W_{k,2}$, then there are only $N$ diffusion
    terms.

  \item Quantum state diffusion (QSD) refers to the choice
    \[\left\{\begin{split}
        a(x) &= \left(-i \ophami + \sum_{k} \Inner{x}{\oplb_k^{\dagger} x} \oplb_k - \frac{1}{2} \oplb_k^{\dagger} \oplb_k - \frac{1}{2} \Abs{\Inner{x}{\oplb_k^{\dagger} x}}^2 \right) x\\
        b_{k,1}(x) &= \frac{1}{\sqrt{2}} \bigl(\oplb_k - \Inner{x}{\oplb_k x} \bigr) x \\
        b_{k,2}(x) &=  \frac{i}{\sqrt{2}} \bigl(\oplb_k - \Inner{x}{\oplb_k x} \bigr) x. \\
      \end{split}\right.\]
    
  \end{itemize}

\addchange{One could easily verify that the above two choices satisfy \eqrefn{eqn::lb_equiv_cond}.}

{\bf Step (\rom{3}):} In \cite{LeBris15}, the authors also considered an
  unraveling scheme resulting from the (deterministic) low-rank
  approximation to \lb{} equations.  As the low-rank
  approximation preserves the structure of the equation, the
  unraveling is similar to the above discussions. This low-rank
  unraveling could be used in control variate for Monte Carlo method
  for \lb{} equations.

\subsection{A commuting diagram for low-rank approximation and
  unraveling}

The route in \cite{LeBris15} is from \lb{} equation to
low-rank quantum master equation and then to the unraveling of
low-rank QME (step (\rom{1}) to step ({3}) in
\figref{fig::diagram}). It is thus natural to consider the alternative
route: \ie, finding the unraveling of \lb{} equation and then applying
\thismethod{} method that we developed in \secref{sec::sdlr} and
\ref{sec::example}); in \figref{fig::diagram}, this refers to the
route from step (\rom{2}) to step (\rom{4}). One immediate question is
whether these two routes commute, in the sense that they end up with
the same equation for low-rank unraveling. Due to the non-uniqueness
of unraveling scheme, in general, the answer is negative.  Perhaps, a
more specific and reasonable question to ask is that given the
unraveling of \lb{} equation, after applying \thismethodshort{}
method, whether its statistical average recovers the low-rank QME
(such as that derived in \cite{LeBris13} using deterministic low-rank approximation).

The answer is positive with slight modification in (deterministic)
dynamical low-rank approximation method in \cite{LeBris13}. See \thmref{thm::lebris} for details of the \reviewnumtwo{modification in constraints} as well as the resulting low-rank QME.  
Moreover,
it turns out that such commuting diagram does not depend on the
unraveling scheme chosen for \lb{} equation. The result is summarized in
\thmref{thm::commuting}.

\medskip

\begin{theorem}[Commuting diagram]
\label{thm::commuting}
For any unraveling scheme in SDE form of \lb{} equation (see
\lemref{lemma::lb_unravel}), the low-rank unraveling  obtained
after applying the \thismethod{} method, 
is an unraveling scheme of
low-rank quantum master equation obtained via (deterministic) dynamical
low-rank approximation method given by \eqrefn{eqn::lebris} in
Theorem~\ref{thm::lebris}.
\end{theorem}

\begin{proof}

By applying \thismethodshort{} (step (\rom{4}) in \figref{fig::diagram}) to unraveling scheme of \lb{} equation
and using \eqrefn{eqn::lb_equiv_cond}, 
one could obtain the following result.
\begin{equation*}
\left\{\begin{split}
\frac{\ud}{\ud t}\opu(t) &= \projperp_{\opu(t)} \lbop\bigl(\ee\bigl[ \opu(t) Y_t Y_t ^{\dagger} \opu(t)^{\dagger} \bigr]\bigr) \opu(t) \ee\bigl[Y_t Y_t^{\dagger} \bigr]^{-1}\\
\ud Y_t &=  \opu^{\dagger}(t) a( \opu(t) Y_t, t)\ \ud t + \sum_{j=1}^{N} \opu^{\dagger}(t) b_j(\opu(t) Y_t, t)\ \ud W_j. \\
\end{split}\right.
\end{equation*}

Next we will verify that the above coupled ODE-SDE system does play the role of unraveling of low-rank QME in \eqrefn{eqn::lebris}, which is step (\rom{3}) in \figref{fig::diagram}. 
By denoting $\opsigma(t) := \ee\bigl[Y_t Y^{\dagger}_t\bigr]$, and $\oprho_{\lr}(t) := \ee\bigl[ X_{\lr,t} X_{\lr,t}^{\dagger} \bigr]\equiv \opu(t) \opsigma(t) \opu^{\dagger}(t)$, one could find that
\begin{equation*}
\begin{split}
\frac{\ud}{\ud t}\opsigma(t) &=  
\ee\left[Y_t a^{\dagger}(\opu(t) Y_t,t) \opu(t) + \opu^{\dagger}(t) a(\opu(t) Y_t, t) Y_t^{\dagger} + \sum_{j=1}^{N} \opu^{\dagger}(t) b_j(\opu(t) Y_t, t) b_j^{\dagger}(\opu(t)Y_t, t) \opu(t) \right]\\
&= \opu^{\dagger}(t)\ee \left[ \bigl(\opu(t) Y_t\bigr) a^{\dagger}(\opu(t) Y_t,t) 
+  a(\opu(t) Y_t, t) \bigl(\opu(t) Y_t\bigr)^{\dagger}
+ \sum_{j=1}^{N}  b_j(\opu(t) Y_t, t) b_j^{\dagger}(\opu(t)Y_t, t)
\right] \opu(t) \\
&= \opu^{\dagger}(t) \ee\left[ \lbop\bigl(\opu(t) Y_t Y_t^{\dagger} \opu^{\dagger}(t)\bigr) \right] \opu(t)
\qquad 
\text{(use \eqrefn{eqn::lb_equiv_cond})}  \\
& = \opu^{\dagger}(t) \lbop\bigl(\opu(t) \opsigma(t) \opu^{\dagger}(t) \bigr) \opu(t).
\end{split}\end{equation*}
The time-evolution equation for $\opu(t)$ can be rewritten, in terms of $\opsigma(t)$, as
\[\frac{\ud}{\ud t} \opu(t) =\projperp_{\opu(t)} \lbop\bigl(\opu(t) \opsigma(t) \opu^{\dagger}(t) \bigr) \opu(t) \opsigma(t)^{-1}. \]
By comparing these two equations with \eqrefn{eqn::lebris}, one could conclude that the low-rank SDE, after applying \thismethodshort{} method, exactly recovers the \eqrefn{eqn::lebris}, which means, the low-rank unraveling is the unraveling for low-rank QME given in \eqrefn{eqn::lebris} and the diagram in \figref{fig::diagram} indeed commutes in this sense. 
Also, note that in the above calculation,
we haven't used any specific choice of unraveling scheme, thus the conclusion is independent of unraveling scheme chosen for \lb{} equation.
\end{proof}

\medskip 

\begin{remark}
  It might not be surprising that the diagram commutes under the above
  conditions. If we consider
  $\oprho_{\lr}(t) = \ee\bigl[X_{\lr,t} X_{\lr,t}^{\dagger}\bigr]$, by
  our conditions, $X_{\lr,t} = \opu(t) Y_t$, hence
  $\oprho_{\lr}(t) = \opu(t) \ee\bigl[Y_t Y_t^{\dagger}\bigr]
  \opu^{\dagger}(t)$,
  which is consistent with the ansatz used in \cite{LeBris13}. What is
  interesting is that the commuting diagram result is independent of
  any unraveling for \lb{} equation, which somewhat shows that the
  low-rank dynamics we derived from the perspective of dynamical
  low-rank approximation in space of signed measures, preserves the
  structure of \lb{} super-operator $\lbop$.
\end{remark}

\reviewnumtwo{
\begin{remark}
Recall that in Theorem \ref{thm::bound}, we have studied how the error between the original SDE and the low-rank approximation from \thismethodshort{} method propagates with respect to time. For the case of unraveling of \lb{} equation,
by using Lemma \ref{lemma::generator} and Lemma \ref{lemma::lb_unravel}, 
	one could show that the growth rate for the error between the unraveling scheme and our corresponding low-rank approximation is bounded by
	\[\gamma(t)  = \Norm{\lbop}_{\hs}:= \sup_{\opf = \opf^{\dagger},\ \Norm{\opf}_{\hs} = 1} \Norm{\lbop(\opf)}_{\hs}. \]
\end{remark}}

\subsection{Discussion on trace-preserving restrictions}\label{sub:trace}

In \cite{LeBris13, LeBris15},  the trace of low-rank approximated density matrix is required to be one, along the time-evolution. However, 
we did not consider such condition above in this section,
nor in the derivation of low-rank dynamics in \thmref{thm::fx=xx} via \thismethodshort{} method. 
One natural question is that what happens if trace-preserving condition is imposed in \thismethodshort{} method.
In our setting up and for the \lb{} equation case,
the trace-preserving constraint in \thismethodshort{} method should be
\begin{equation*}
\tr\bigl(\oprho_{\lr}(t)\bigr) \equiv \tr\left(\opu(t) \ee_{\theta_t}\bigl[y y^{\dagger}\bigr] \opu^{\dagger}(t)\right) = 1, \qquad \text{or equivalently} \qquad \tr\left(\ee_{\theta_t}\bigl[yy^{\dagger}\bigr]\right) = 1,
\end{equation*}
because $\opu^{\dagger}(t) \opu(t) = \id_{r\times r}$.
From the constraint that $\frac{\ud }{\ud t} \tr\left(\ee_{\theta_t}\bigl[yy^{\dagger}\bigr]\right) = 0$, one could derive that 
\begin{equation}
\label{eqn::trace_constraint}
 \tr\left( \ee_{\theta_t}\bigl[ y A(y, t)^{\dagger} + A(y, t) y^{\dagger} + \sum_{j=1}^{M} B_j(y,t) B_j^{\dagger}(y,t) \bigr]   \right) = 0.
\end{equation}

Recall that in the proof of \thmref{thm::fx=xx}, such constraint does
not affect the first-order stationary condition with respect to
$\opg(t)$, hence $\opg(t)$ still has the same form. Thus, the
optimization problem in the derivation is still
\begin{equation*}
\min_{A,\ \bigl\{B_j\bigr\}_{j=1}^{M} } \Inner{\proj_{\opu(t)} \opc \proj_{\opu(t)}}{\proj_{\opu(t)} \opc \proj_{\opu(t)}},
\end{equation*}
where 
\begin{equation*}
\begin{split}
\proj_{\opu(t)} \opc \proj_{\opu(t)} &= \proj_{\opu(t)} \ee_{\theta_t} \bigl[ \bigl(\opu(t) A(y,t) - a(\opu(t)y,t)\bigr)y^{\dagger} \opu^{\dagger}\bigr] + \hc + \sum_{j=1}^{M} \opu(t) \ee_{\theta_t} \bigl[B_j(y,t) B_j^{\dagger}(y,t)\bigr] \opu^{\dagger}(t) \\
& \qquad  - \sum_{j=1}^{N}\proj_{\opu(t)} \ee_{\theta_t} \bigl[b_j(\opu(t)y,t) b_j^{\dagger}(\opu(t)y,t)\bigr] \proj_{\opu(t)} \\
&= \opu(t) \ee_{\theta_t}\left[ A(y,t) y^{\dagger} + y A^{\dagger}(y,t) + \sum_{j=1}^{M} B_j(y,t) B_j^{\dagger}(y,t) \right] \opu^{\dagger}(t) \\
& \qquad - \proj_{\opu(t)} \ee_{\theta_t}\left[ a(\opu(t) y, t) (\opu(t) y)^{\dagger} + (\opu(t) y) a^{\dagger}(\opu(t) y, t) + \sum_{j=1}^{N} b_j(\opu(t) y, t) b_j^{\dagger}(\opu(t) y, t) \right]\proj_{\opu(t)} \\
&= \opu(t) \ee_{\theta_t}\left[ A(y,t) y^{\dagger} + y A^{\dagger}(y,t) + \sum_{j=1}^{M} B_j(y,t) B_j^{\dagger}(y,t) \right] \opu^{\dagger}(t)  - \proj_{\opu(t)} \lbop\left(\opu(t) \ee_{\theta_t}\bigl[yy^{\dagger}\bigr]\opu^{\dagger}(t)\right) \proj_{\opu(t)}.
\end{split}
\end{equation*}
In the last step, we have used \eqrefn{eqn::lb_equiv_cond}. 

If we assume it is possible to achieve
$\proj_{\opu(t)} \opc \proj_{\opu(t)} = 0$, then
\begin{equation*}
\ee_{\theta_t}\left[ A(y,t) y^{\dagger} + y A^{\dagger}(y,t) + \sum_{j=1}^{M} B_j(y,t) B_j^{\dagger}(y,t) \right]   = \opu^{\dagger}(t) \lbop\left(\opu(t) \ee_{\theta_t}\bigl[yy^{\dagger}\bigr]\opu^{\dagger}(t)\right) \opu(t).
\end{equation*}
The constraint in \eqrefn{eqn::trace_constraint} 
requires the the trace of left hand side is zero, while on the right hand side, in general, 
\begin{equation*}
\tr\left\{ \opu^{\dagger}(t) \lbop\left(\opu(t) \ee_{\theta_t}\bigl[yy^{\dagger}\bigr]\opu^{\dagger}(t)\right) \opu(t)\right\} = \tr\left\{ \proj_{\opu(t)}  \lbop\left(\opu(t) \ee_{\theta_t}\bigl[yy^{\dagger}\bigr]\opu^{\dagger}(t)\right) \right\} \neq 0,
\end{equation*}
even though $\tr\left\{ \lbop\left(\opu(t) \ee_{\theta_t}\bigl[yy^{\dagger}\bigr]\opu^{\dagger}(t)\right) \right\} \equiv 0$. One could conclude that in general, the minimization problem 
\begin{equation*}
\min_{A,\ \bigl\{B_j\bigr\}_{j=1}^{M} } \Inner{\proj_{\opu(t)} \opc \proj_{\opu(t)}}{\proj_{\opu(t)} \opc \proj_{\opu(t)}} > 0, 
\end{equation*}
under the constraint in \eqrefn{eqn::trace_constraint}.
This could be anticipated since one has to work on a smaller space during minimization. 
Hence, in general, the low-rank unraveling from \thismethodshort{} with trace-preserving constraint is not optimal in the sense of \thmref{thm::fx=xx}.

In fact, from the optimization problem,
it is not straightforward how to minimize the $\Inner{\proj_{\opu(t)} \opc \proj_{\opu(t)}}{\proj_{\opu(t)} \opc \proj_{\opu(t)}}$ under the trace constraint for functions $A$ and $B_j$. After all, we need to work on the quadratic variational problem on functional space with trace constraint, not on matrices as in \cite{LeBris13}. 

While it is perhaps desirable to have a trace-preserving dynamical
low-rank approximation of the density matrix, 
\reviewnumtwo{our choice of not considering trace-preserving constraint can be justified via better approximating measurement outcome.} 
Since the expected
measurement outcome for observable $\opo$ is
$\tr(\opo \oprho) = \Inner{\opo}{\oprho}$, to get an accurate
approximation of the expectation, it is sufficient that
$\oprho_{\lr}(t)$ is close to $\oprho(t)$ in Hilbert-Schmidt norm,
without the requirement of trace-preserving constraint. One could also
normalize the resulting density matrix from the \thismethodshort{} as
a postprocessing step.

\subsection{Methods selection and control variate}

From the commuting diagram, one might question the usefulness of low-rank unraveling in practice. Since in the low-rank unraveling, at each time step, one needs to store $\mathcal{O}(\dimn\times r  + N_s\times r)$ data for $\opu(t)$ and random variable $Y$, where $N_s$ is the sample size. For the deterministic low-rank dynamics, it only needs to store $\mathcal{O}(\dimn\times r + r^2)$ for $\LU(t)$ and $\LSIG(t)$.

When $r = \mathcal{O}(1)$ is useful to approximate the full dynamics,
solving the deterministic low-rank approximation of \lb{}
equation is a better choice as one does not need to simulate many
sample paths to get statistical averages and the simulation of
deterministic low-rank dynamics has smaller memory and computational cost.

On the other hand, when rank $r$ requires to be large in order to
approximate the system accurately, then it becomes inefficient to
solve the deterministic low-rank approximation for \lb{} equation. It
is advantageous to turn to stochastic approximation and to use
unraveling of \lb{} equation.

However, one could consider using control variate method \cite{LeBris15} to facilitate the simulation, that is, to use
\begin{equation*}
\overline{\oprho} = \overline{\oprho}_{MC} + \lambda({\oprho}_{\lr}  - \overline{\oprho}_{\text{LRMC}}), 
\end{equation*}
where $\overline{\oprho}_{MC}$ is obtained via unraveling scheme;
$\oprho_{\lr}$ is obtained by solving deterministic low-rank dynamics
of \lb{} equation; $\overline{\oprho}_{\text{LRMC}}$ is obtained by
solving low-rank unraveling scheme. The rank $r = \mathcal{O}(1)$ to
ensure that one could simulate both deterministic and stochastic
low-rank dynamics. The low-rank dynamics cannot provide accurate
approximation, however, one could still simulate low-rank dynamics to
achieve reduction of variance, by choosing correct parameter
$\lambda$. Please refer to \cite{LeBris15} for details and numerical performance.

Generally speaking, for Fokker-Planck equation and \lb{} equation
with continuum state space (e.g., probability on $\mathbb{R}^d$ or
density matrix over $L^2(\mathbb{R}^d)$), one would not want to solve
a deterministic low-rank dynamics for them directly, since that is
still a PDE in potentially high dimension when $r$ is not so small
(even when $r = 6$, solving such a PDE in $6$ dimension is already
rather challenging with standard methods). When the low-rank
approximation is accurate, one could choose to use low-rank SDE
(derived by \thismethodshort{} method) to achieve reduction of
complexity in model; otherwise, one has to simulate the original
SDE. Of course, this discussion only involves which model to solve;
detailed numerical methods and algorithmic implementation would still
make a significant difference to the overall performance.


\section{Numerical experiments}
\label{sec::numerics}

In this section, we will validate our method using numerical examples
of some high-dimensional SDEs: high-dimensional geometric Brownian
motion, stochastic Burgers' equation 
and unraveling of quantum damped
harmonic oscillator.  For the first two examples, we will also compare
\thismethodshort{} method with DO method.  It could be observed that
\thismethodshort{} method has comparable performance in approximating
the mean when the rank is chosen correctly, compared with DO method;
and it performs better in approximating the second moment.
We will measure the relative error as the indicator
of performance of low rank approximation: 
\begin{itemize}
\item the relative error for the
mean (linear) is defined as
$\Norm{\ee_{\mu_t} \bigl[x\bigr] -
  \ee_{\mu_{\lr,t}}\bigl[x\bigr]}/\Norm{\ee_{\mu_t} \bigl[x\bigr]}$\,;
\item
the relative error for the second moment (quadratic) is defined
as
$\Norm{ \ee_{\mu_{t}} \bigl[xx^{\dagger}\bigr] -
  \ee_{\mu_{\lr,t}}\bigl[
  xx^{\dagger}\bigr]}_{\hs}/\Norm{\ee_{\mu_t}\bigl[xx^{\dagger}\bigr]}_{\hs}$\,.
\end{itemize}

Let us comment on some details of the numerical implementation. 
For simplicity, Euler-Maruyama method is used as stochastic
integrator.  An order-one deterministic numerical scheme in
\cite{Higham96} is employed to preserve orthogonality of $\opu(t)$. The
inverse of $\ee_{\theta_t} \bigl[ yy^{\dagger}\bigr]$ in
\eqrefn{eqn::low_rank_xx} will cause numerical instability when its
condition number is large. This problem also appears in dynamical
orthogonal method \cite{Sapsis09}. Paper \cite{Babaee17} suggested to use pseudo-inverse to maintain the
algorithmic stability, which is also adopted here in the numerical
simulation.  If the chosen rank is representative (not over-estimating
the rank), then the pseudo-inverse should be simply the matrix
inverse.

\subsection{Geometric Brownian motion}

Consider the geometric Brownian motion of the form
\begin{equation*}
\ud X_t = \vect{\Lambda} X_t\ud t + \vect{\Theta} X_t\ud W_t, 
\end{equation*} 
where $X_t \in \Complex^{20}$, $\vect{\Lambda}$,
$\vect{\Theta} \in \Complex^{20\times 20}$.

A rank-5 initial condition is used with $X_{0} = x_k$ with probability
$p_k$ where $\{x_k\}_{k=1}^{5}$ are randomly generated orthogonal
vectors in $\Complex^{20}$ and $p_k \propto  \text{Poisson}(k-1,0.5)$ for
$1\le k \le 5$. $\text{Poisson}(k,\lambda)$ represents the probability
density function of Poisson distribution with rate $\lambda$ at value
$k$. $\vect{\Lambda}$ is of the form $\vect{Q} \opd \vect{Q}^{\dagger}$ where $\opd$ is a
randomly generated diagonal matrix with diagonal elements uniformly
distributed in the interval $[-4.5, -0.5]$ and $\vect{Q}$ is a randomly
generated orthogonal matrix;
$\vect{\Theta} = \sqrt{0.05}\, \id_{20 \times 20}$.  It is not difficult to
prove that for such geometric Brownian motion, the second moment
$ \ee_{\mu_t} \bigl[xx^{\dagger}\bigr]$ decays to $0$ as
$t\rightarrow\infty$ (known as mean-square stability).  The result is
visualized in Figures~\ref{fig::gbm_spectrum}, \ref{fig::gbm_SDLR}  for both \thismethodshort{} method and DO
method.  It should be remarked that to faithfully compare
\thismethodshort{} with DO under the same rank $r$, \eg, $r = 5$, the
matrix $\opu(t)$ in \thismethodshort{} has dimension $20\times 5$,
whereas it has dimension $20\times 4$ for DO method since $\bar{X}$
for DO method should account for one rank and contributes $n=20$
degrees of freedom. 

In numerical experiment, the sample size is $10^5$ and time step is $\frac{1}{300}$. From \figref{fig::gbm_spectrum}, the mean-square stability is clearly observed. In \figref{fig::gbm_SDLR}, for \thismethodshort{} method, when rank increases, the relative error decreases for both mean and second moment. Since the fifth eigenvalue is extremely small compared with others, it is reasonable that choosing rank $5$ does not significantly improve the accuracy further. 
 For DO method, even for small rank, the relative error of mean  is small; the relative error for the second moment decreases as rank increases, which is expected.
As can be seen, DO method captures the mean better 
and \thismethodshort{} method captures the second moment better. 
This finding is consistent with the theoretical derivation.

\begin{figure}[h!]
\centering
\includegraphics[width=0.48\textwidth]{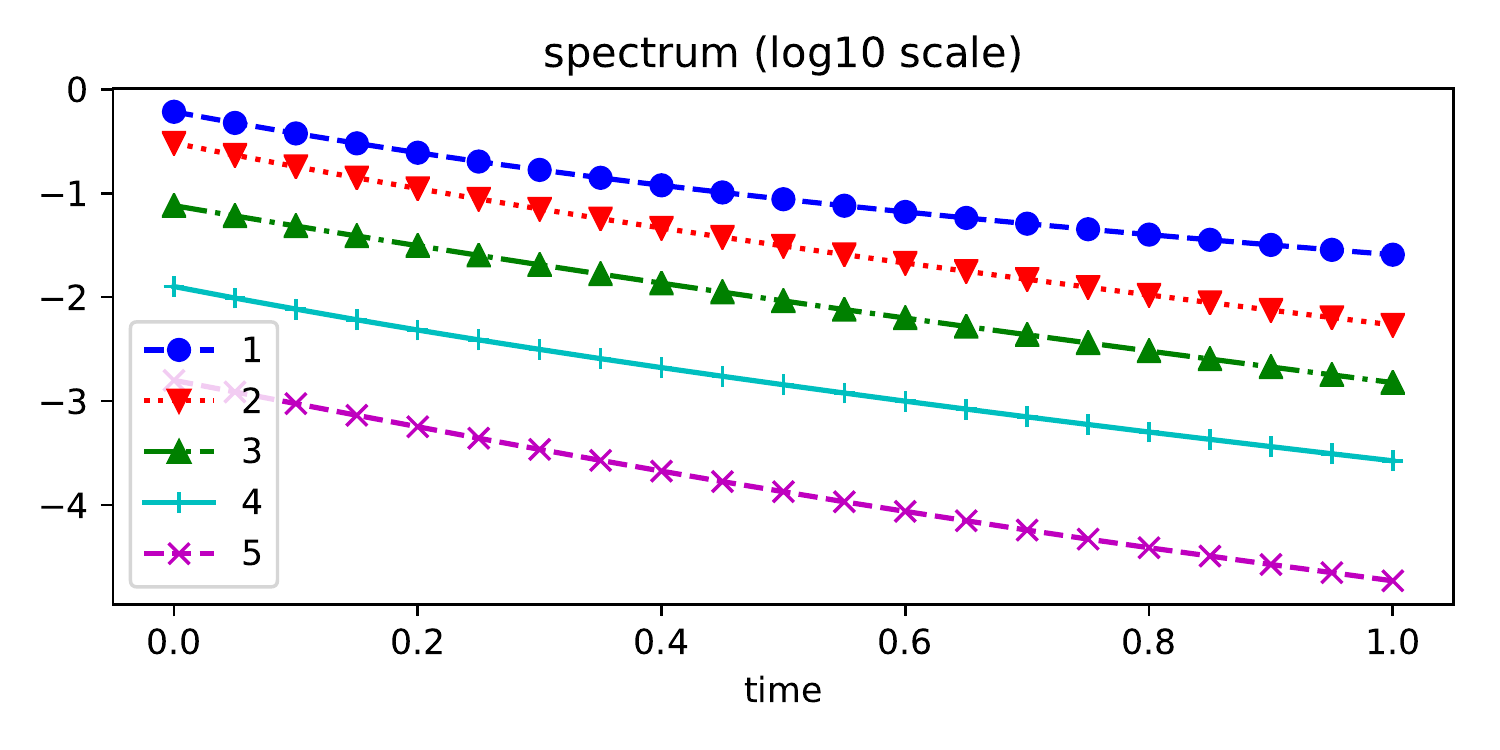}

\caption{Spectrum of $\ee_{\mu_t} \bigl[xx^{\dagger}\bigr]$ in $\log 10$ scale (the five largest eigenvalues) for high-dimensional geometric Brownian motion.}
\label{fig::gbm_spectrum}
\end{figure}

\begin{figure}[h!]
\centering
\begin{subfigure}[b]{0.48\textwidth}
\includegraphics[width=\textwidth]{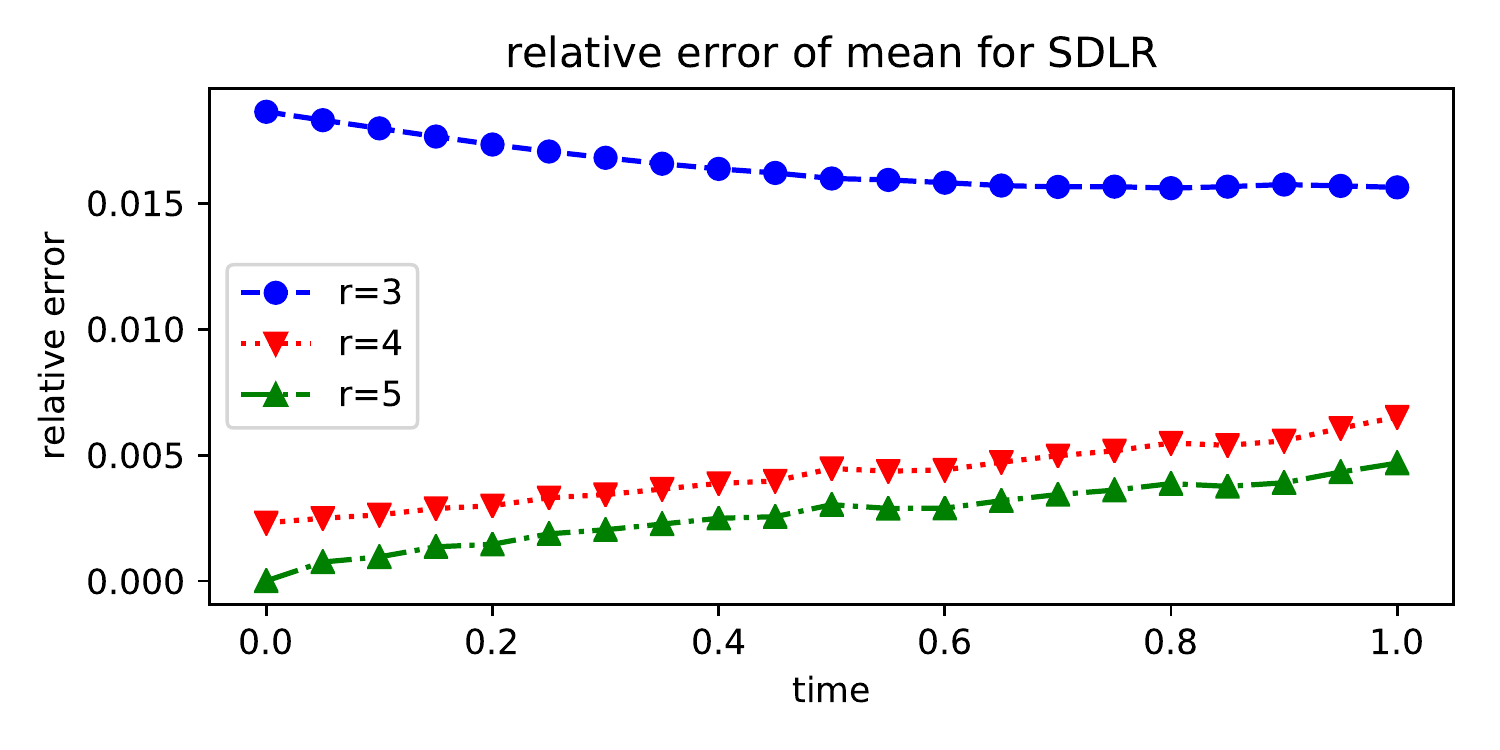}
\end{subfigure}
~
\begin{subfigure}[b]{0.48\textwidth}
\includegraphics[width=\textwidth]{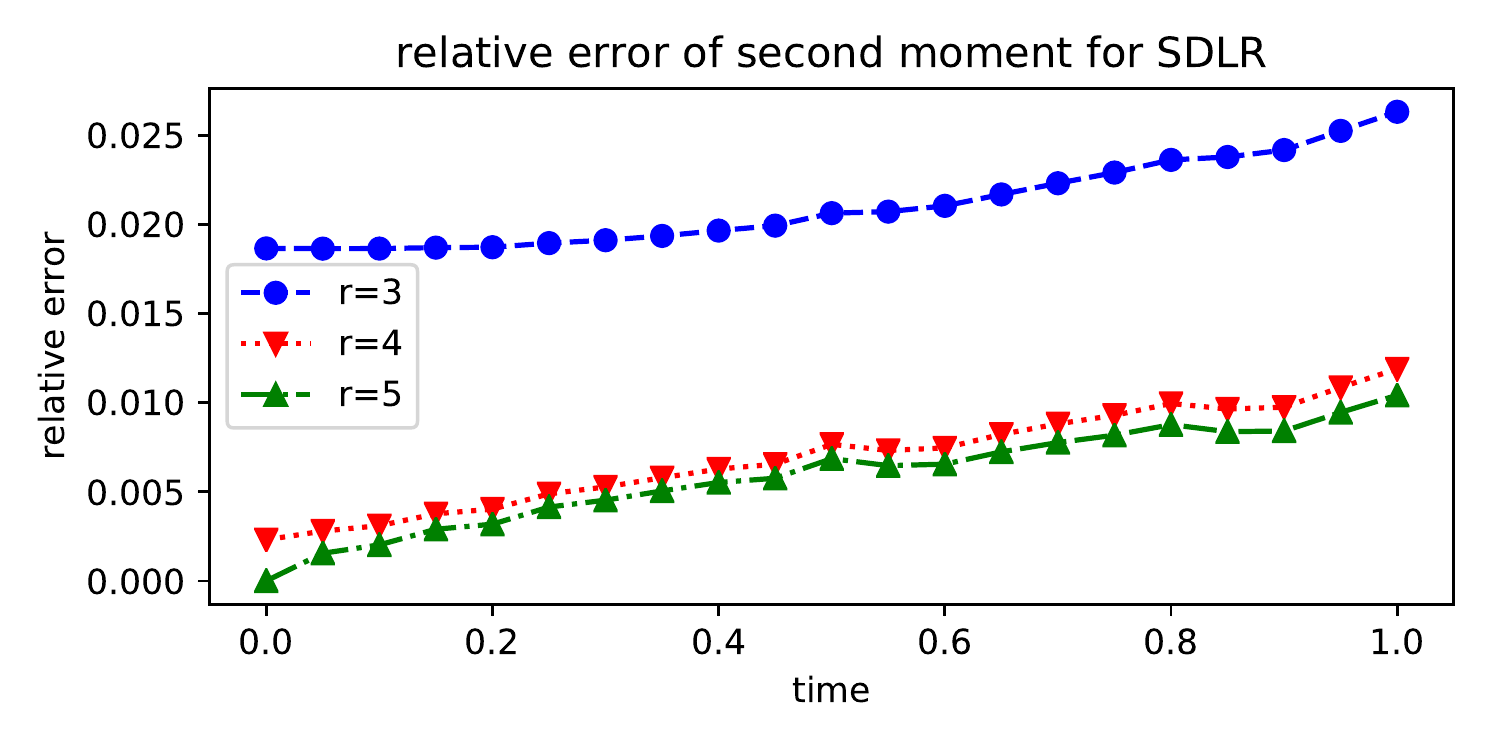}
\end{subfigure}
~
\begin{subfigure}[b]{0.48\textwidth}
\includegraphics[width=\textwidth]{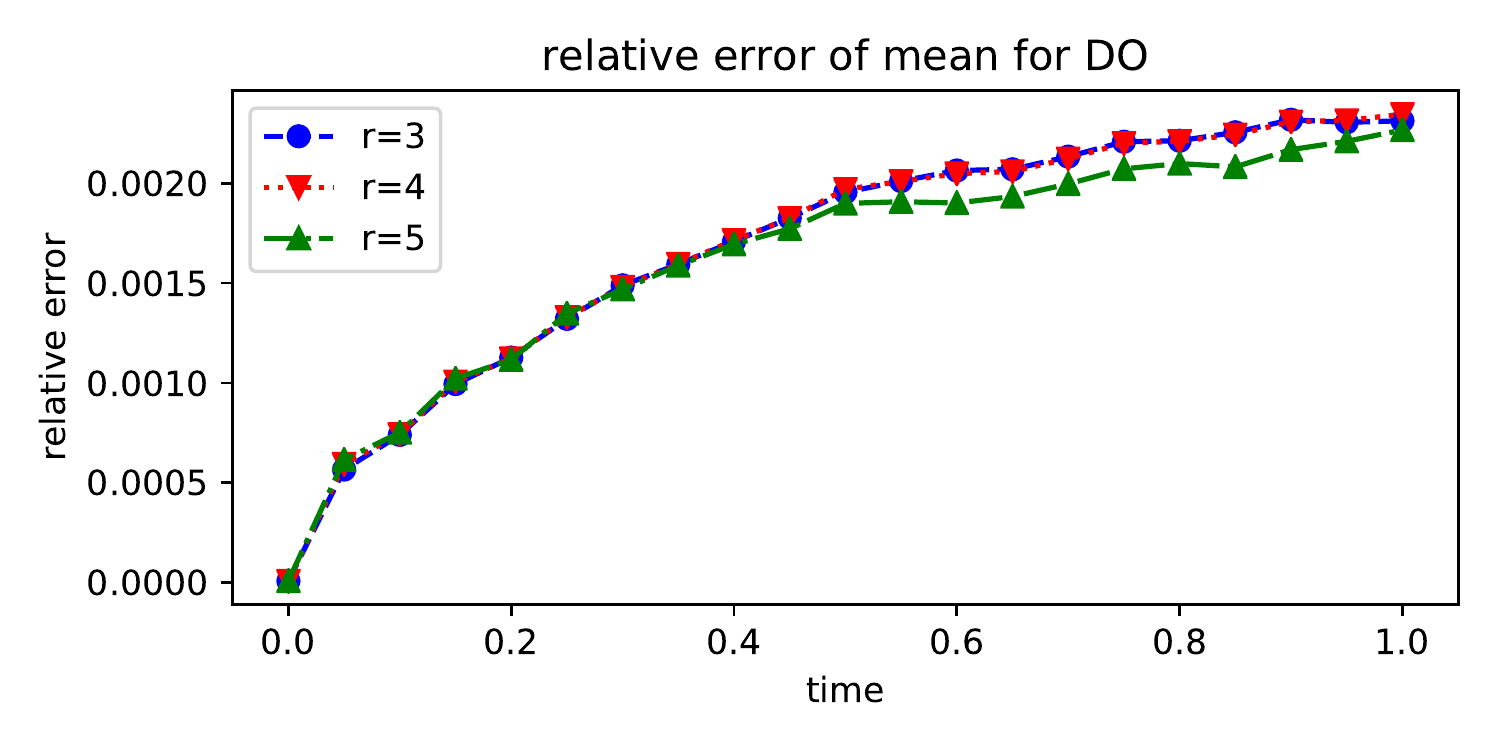}
\end{subfigure}
~
\begin{subfigure}[b]{0.48\textwidth}
\includegraphics[width=\textwidth]{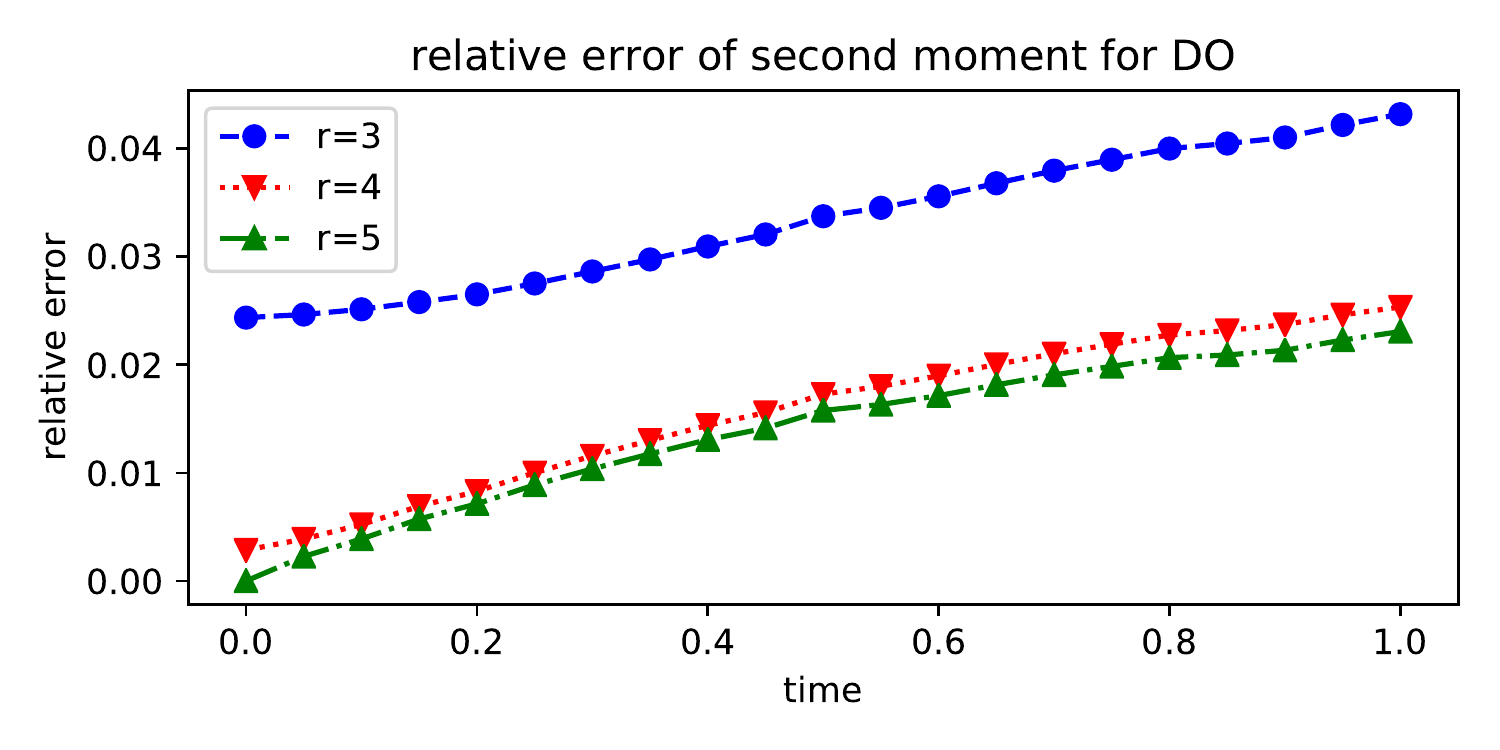}
\end{subfigure}

\caption{Relative error for \thismethodshort{} method and DO method in solving high-dimensional geometric Brownian motion.}
\label{fig::gbm_SDLR}
\end{figure}

\subsection{Stochastic Burgers' equation}

We will use the stochastic Burgers' equation $h(z, t,\omega)$ of the form
\begin{equation}
\begin{split}
& \ud h = (\nu \partial_{z}^2 h - h \partial_z h) \ud t + g(z) \ud W \\
& h(z, 0, \omega) = h_0(z, \omega) \qquad h(0, t, \omega) = h(1, t, \omega) ,
\end{split}
\end{equation}
for $(z, t) \in [0,1]\times [0,T]$. This example is adapted from \cite{Ozen16_SPDE}. Notice that $\ud W$ is chosen as a scalar Brownian motion, independent of spatial coordinate $z$.

Due to the periodic boundary condition, by separating $(t, \omega)$ from variable $z$,
\[ h(z,t,\omega) = \sum_{k\in \Int} X_k(t,\omega) e^{2\pi i k z}. \]
Then stochastic Burgers' equation could be viewed as a SDE on Hilbert space $L^2[0,1]$ with basis functions $\{e^{2\pi i k z} \}_{k\in \Int}$. The SDE has the form
\begin{equation}
\begin{split}
\ud X_{k}(t,\omega) &= \left(- (2\pi k)^2 \nu X_{k}(t,\omega) - \sum_{k'} X_{k-k'}(t,\omega) X_{k'}(t,\omega) (2\pi i k') \right)\ \ud t + g_k\ \ud W \\
X_{k}(0,\omega) &= \Inner{e^{2\pi i k z}}{h_0(z,\omega)} \qquad g_k = \Inner{e^{2\pi i k z }}{g(z)},
\end{split}
\end{equation}
where the inner product $\Inner{f_1}{f_2} = \int_{0}^{1} f_1^{*}(z) f_2(z)\ \ud z$ and $\ud W$ is independent of the mode $k$.

In the direct numerical simulation, we truncate $k$ by letting $X_{k}(t,\omega) = 0$ for $|k|> \frac{n-1}{2}$ where $n$ is an odd positive integer.  Then $X_{k}(t,\omega)$ can be stored in a $\Complex^{n}$ vector. 
With careful choice of initial condition and let $n\rightarrow \infty$, we would expect to have the solution of such truncated SDE converge to the true solution. 
Adapted from Example 4.1 in Ref.~\cite{Ozen16_SPDE}, let $\nu = 0.01$, 
and \[g(z) = \gamma \cos(2\pi z) = \frac{\gamma}{2} \left(e^{2\pi i z} + e^{2\pi i (-1) z}\right), \qquad \gamma = \frac{1}{10}. \]

As for initial condition, a rank-5 case is considered
\[h_0(z, \cdot) = \left\{ \begin{split}
1, &\qquad  \text{with probability } p_1 \\
\sqrt{2} \sin(2\pi \floor{k/2} z), &\qquad \text{with probability } p_k,\ k = 2, 4\\
\sqrt{2} \cos(2\pi \floor{k/2} z), &\qquad \text{with probability } p_k,\ k = 3, 5 \\ 
\end{split}  \right.\]
where $p_k\propto \text{Poisson}(k-1,0.5)$ just like last example and $\floor{ \cdot }$ is the floor function.

In \figref{fig::sbe_dimension}, it could be observed that the numerical result is stable with respect to dimension for truncation, \ie, $n$. Therefore, it is justifiable to simply solve the truncated system with $n = 21$ by \thismethodshort{} and DO method.

\begin{figure}[h!]
\centering
\includegraphics[width=0.48 \textwidth]{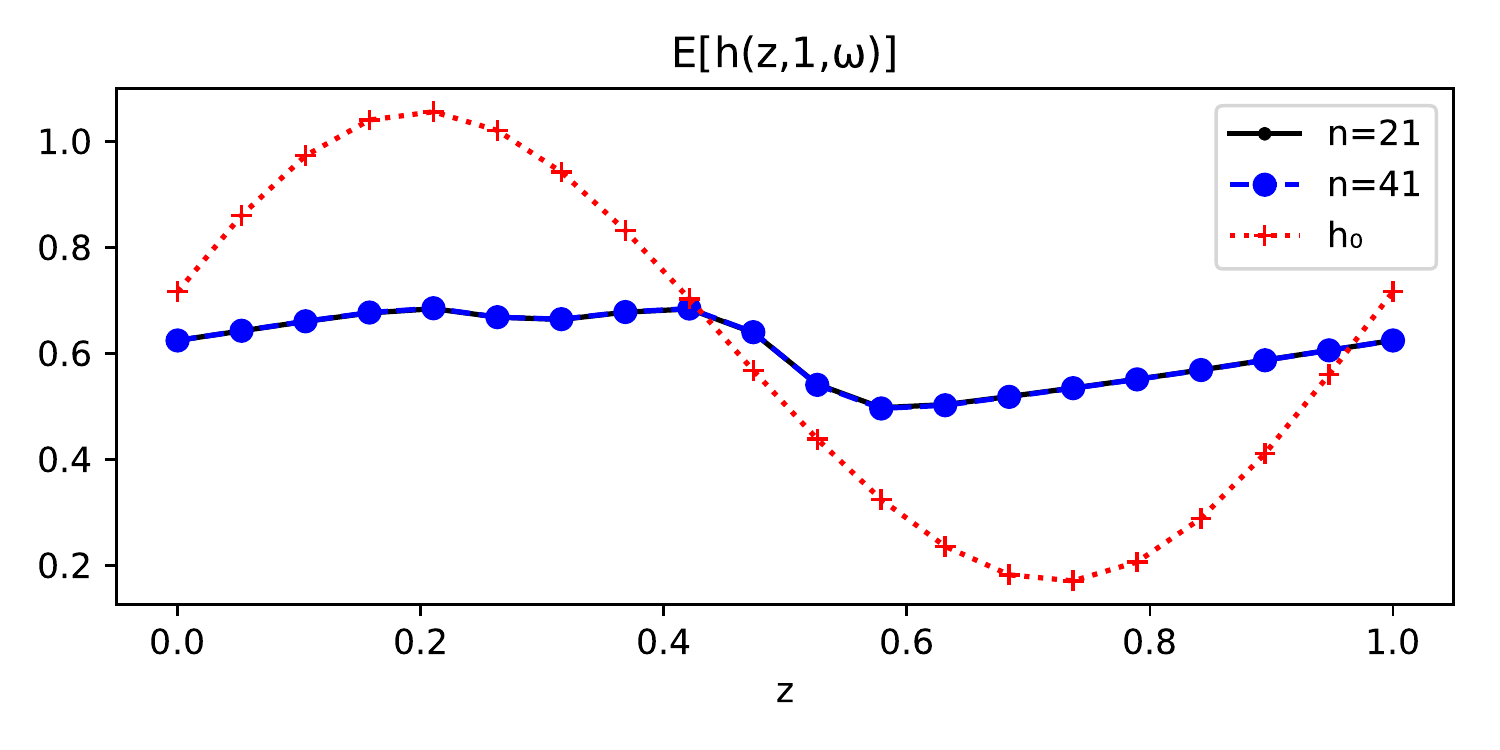}
\caption{This figure shows the expected function $\ee[h(z,1, \omega)]$ on $z\in [0,1]$, for different truncated dimension $n$, in solving stochastic Burgers' equation.}
\label{fig::sbe_dimension}
\end{figure}

For $n = 21$, sample size $10^4$, time step $\frac{1}{200}$, the spectrum and relative error are visualized in Figures \ref{fig::sbe_spec} and \ref{fig::sbe_error} respectively. From the spectrum, the second  moment tends to behave like a rank one matrix, since the largest eigenvalue almost keeps a constant while other eigenvalues roughly exponentially decay. 
In \figref{fig::sbe_error}, for both methods, when rank increases, the relative error decreases, 
which is consistent with expectation.
It could be seen that 
the performance of \thismethodshort{} method and DO method is similar and comparable,  
in calculating both $\ee_{\mu_t} \bigl[x\bigr]$ and $\ee_{\mu_t} \bigl[xx^{\dagger}\bigr]$. 
As for more detailed performance comparison between these two methods, \figref{fig::sbe_error} is not very informative, especially for second moment. 
Therefore, we additionally provide Table \ref{table::error_SBE}, 
which gives the ratio between relative error for \thismethodshort{} method and relative error for DO method at time $t = 1$, for various rank $r$. 
Then, it could be observed that DO method performs better in approximating the mean and \thismethodshort{} method performs better in approximating the second moment.
As a reminder, the result in Table \ref{table::error_SBE}  can only be interpreted qualitatively, due to random fluctuation in simulation.

\begin{figure}
\centering
\includegraphics[width=0.48\textwidth]{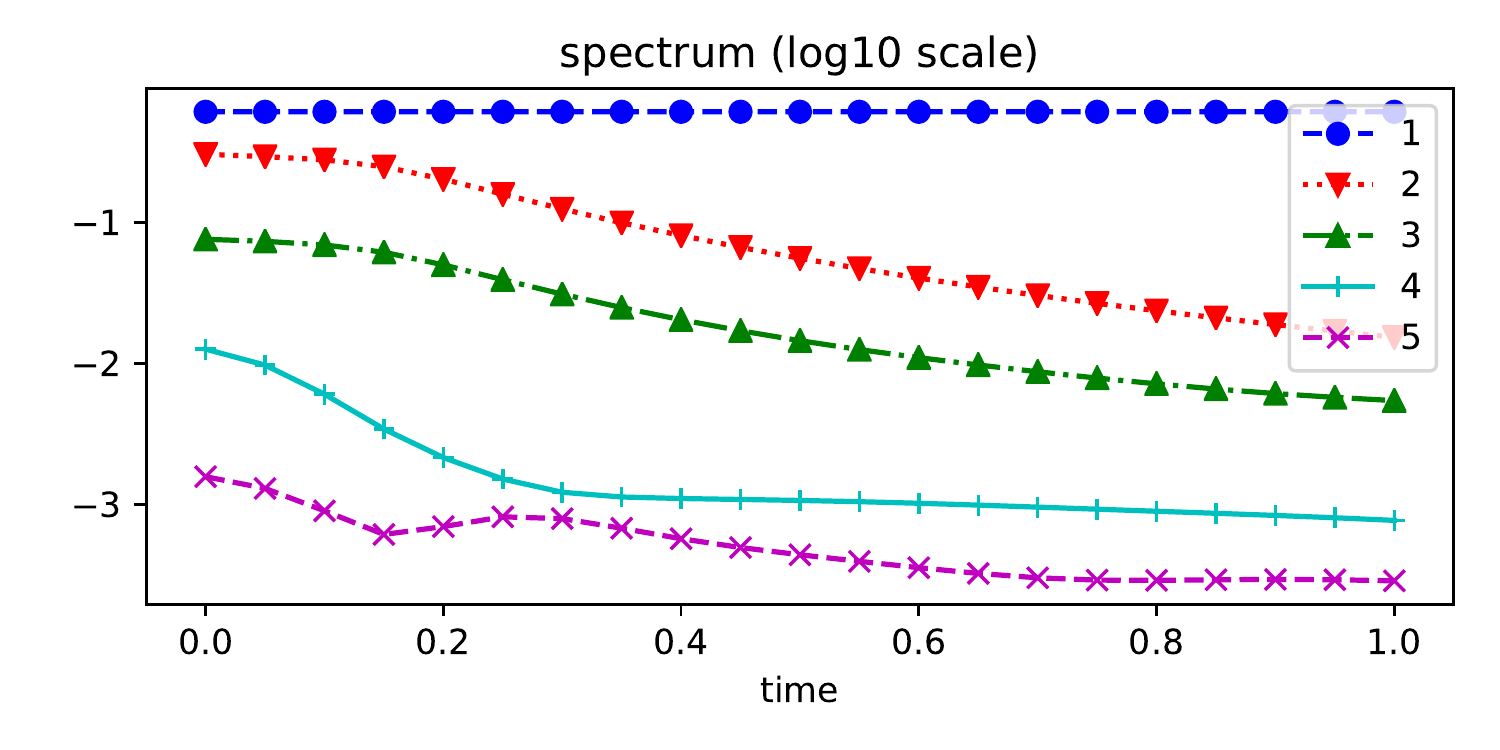}
\caption{Spectrum of $\ee_{\mu_t} \bigl[xx^{\dagger}\bigr]$ in $\log 10$ scale (the five largest eigenvalues) for stochastic Burgers' equation.}
\label{fig::sbe_spec}
\end{figure}

\begin{figure}
\centering
\begin{subfigure}[b]{0.48\textwidth}
\includegraphics[width=\textwidth]{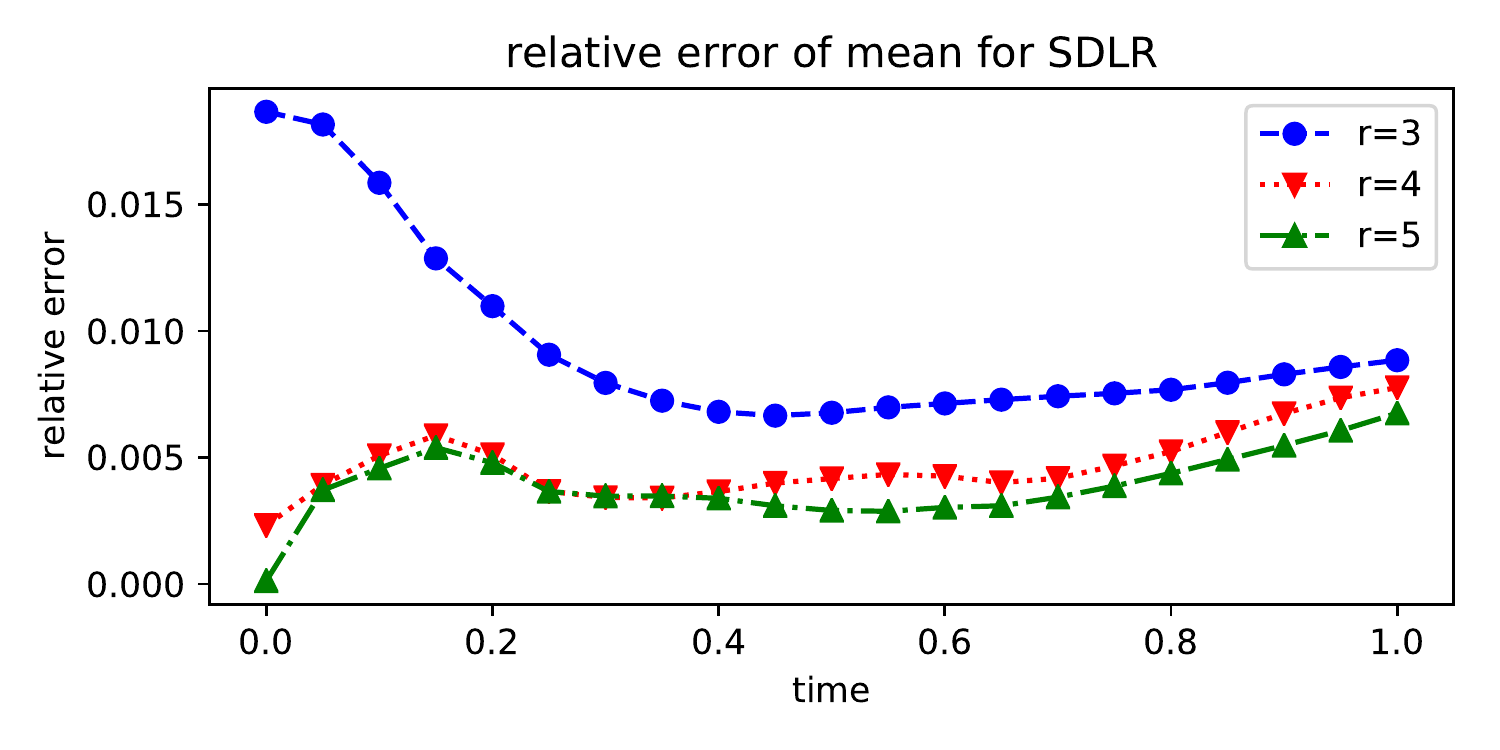}
\end{subfigure}
~
\begin{subfigure}[b]{0.48\textwidth}
\includegraphics[width=\textwidth]{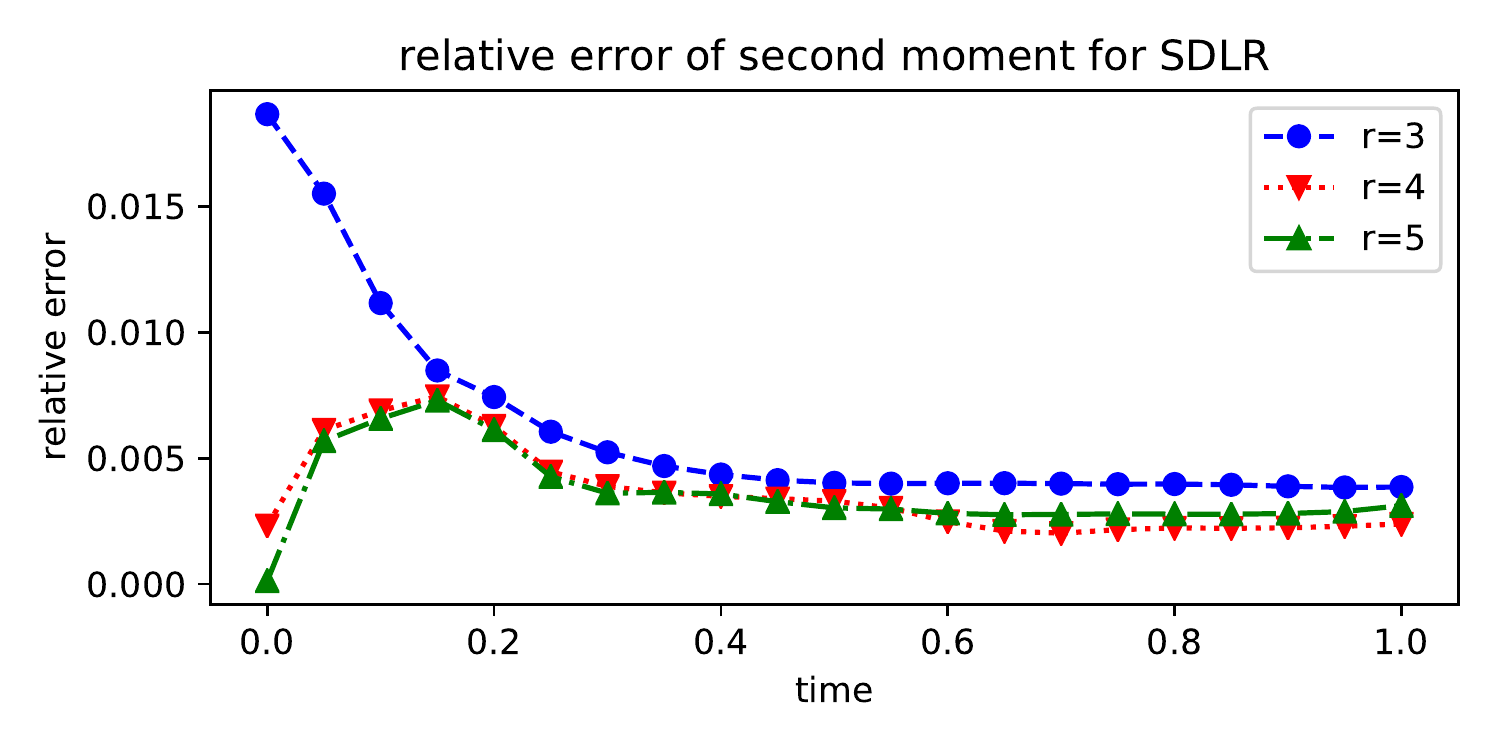}
\end{subfigure}
~
\begin{subfigure}[b]{0.48\textwidth}
\includegraphics[width=\textwidth]{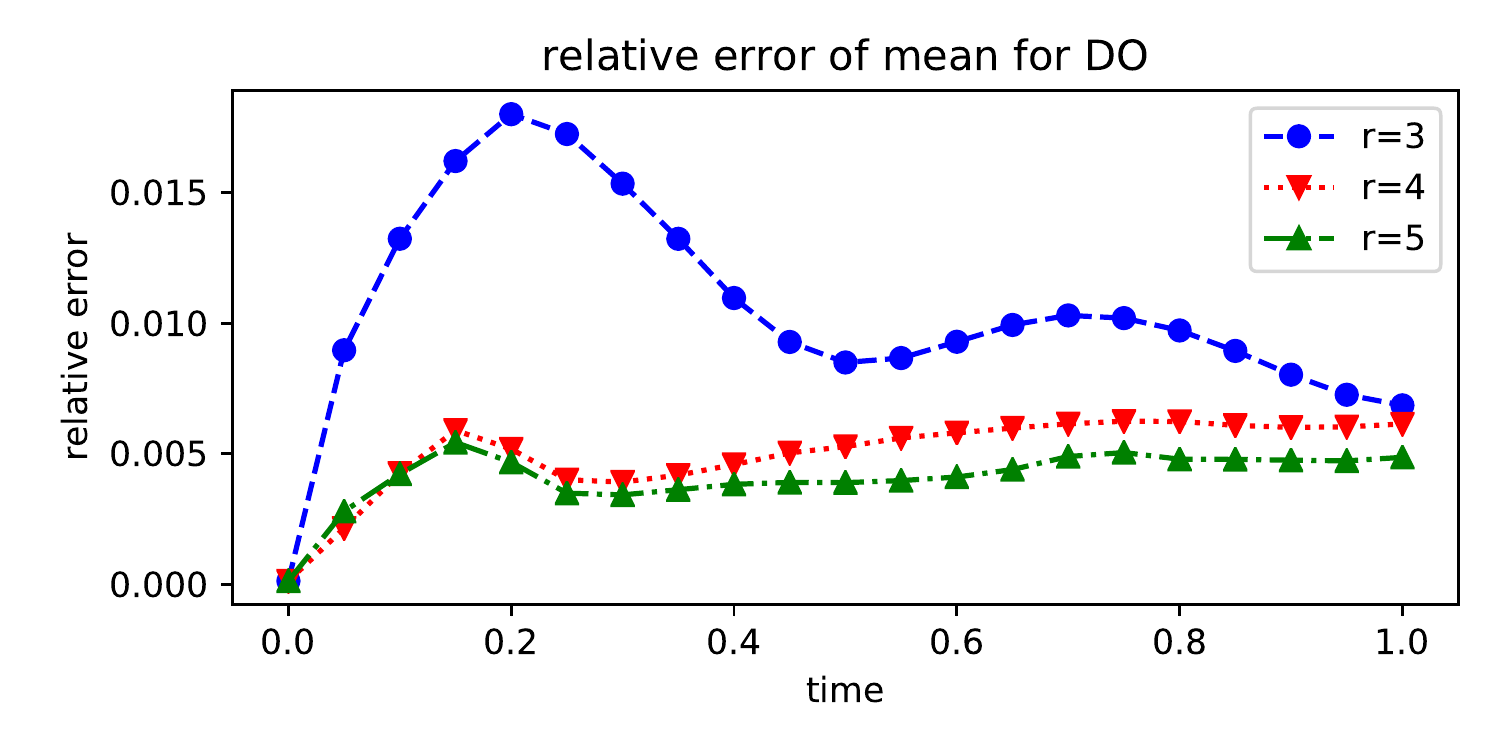}
\end{subfigure}
~
\begin{subfigure}[b]{0.48\textwidth}
\includegraphics[width=\textwidth]{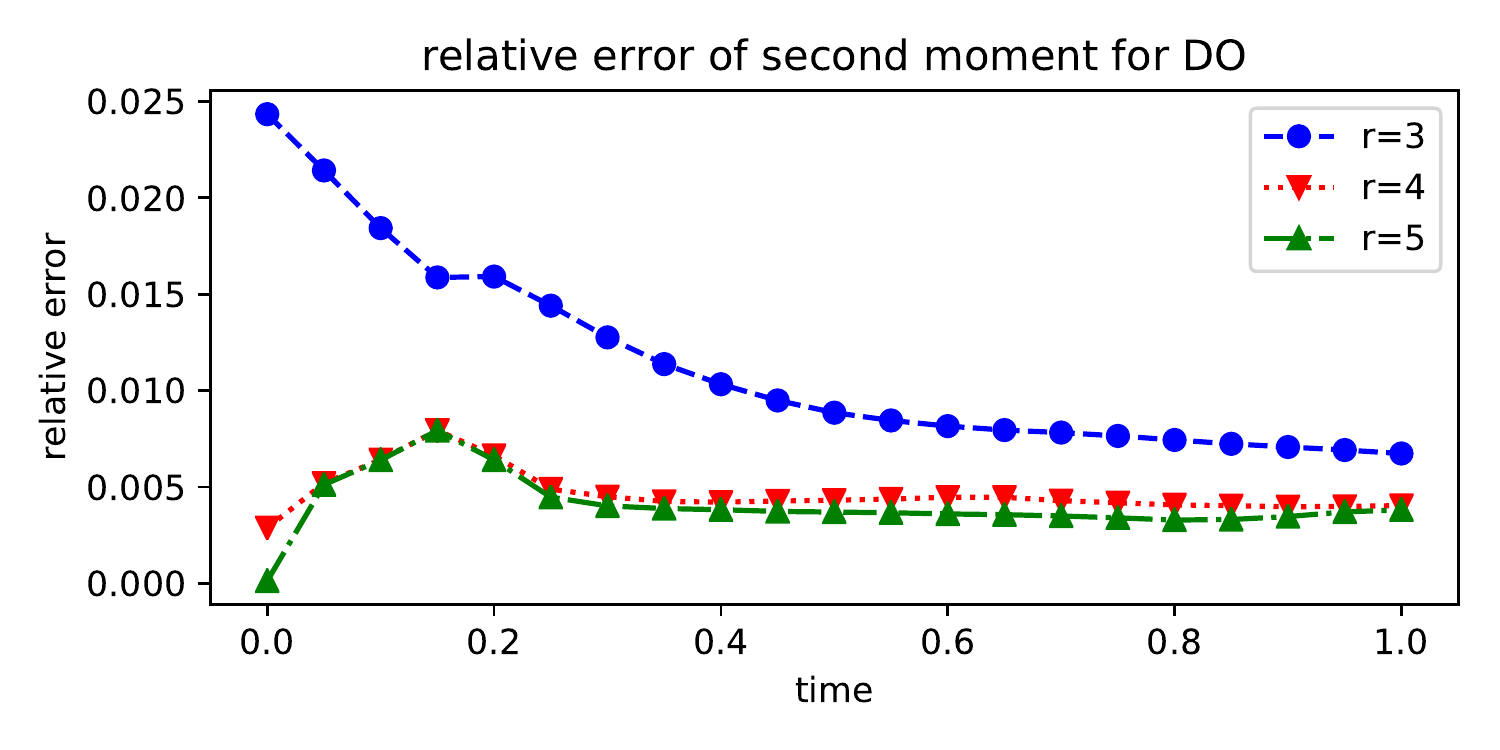}
\end{subfigure}
\caption{Relative error for \thismethodshort{} method and DO method in solving stochastic Burgers' equation.}
\label{fig::sbe_error}
\end{figure}

\begin{table}
\centering
\begin{tabular}{c|c c}
\toprule
$r$ & ratio for relative error of $\ee_{\mu_t}\big[x\bigr]$ & ratio for relative error of $\ee_{\mu_t} \big[xx^{\dagger}\bigr]$ \\
\toprule
3	&   1.291	&  0.572 	\\
4	&   1.264	&  0.593	\\
5	&   1.389	&  0.816	\\
\bottomrule
\end{tabular}
\caption{This table summarizes the ratio between  relative error for \thismethodshort{} method and the relative error for DO method, at time $t = 1$ for stochastic Burgers' equation.}
\label{table::error_SBE}
\end{table}

\subsection{Quantum damped harmonic oscillator}

This example is to solve a simple quantum damped harmonic oscillator, see \eg,  \cite{Fujii12}. The \lb{} equation is given by
\begin{equation}
\dot{\oprho} = -i \Comm{ \omega \vect{d}^{\dagger} \vect{d}}{ \oprho } +
 \gamma_1 \left( \vect{d} \oprho \vect{d}^{\dagger} - \frac{1}{2} \Anticomm{\vect{d}^{\dagger} \vect{d}}{ \oprho}  \right) + 
 \gamma_2 \left( \vect{d}^{\dagger}\oprho \vect{d} - \frac{1}{2} \Anticomm{\vect{d} \vect{d}^{\dagger}}{\oprho} \right),
\end{equation}
where $\omega$ is angular frequency, $\vect{d}^{(\dagger)}$ are annihilation (creation) operator for harmonic oscillator. 
Adopting the bra-ket notation, let $\left\{\ \Ket{k}\ \right\}_{k=0, 1, \cdots, n-1}$ be the orthonormal basis of quantum states. Then the effect of operator $\vect{d}^{(\dagger)}$ is $\vect{d}^{\dagger}\ket{k} = \sqrt{k+1}\ket{k+1}$ and $\vect{d}\ket{k} = \sqrt{k} \ket{k-1}$. The special state $\Ket{0}$ usually refers to ground state and $\Ket{k}$ ($k\ge 1$) are known as excited states, in physical literatures.

We truncate the system by $n=21$ states, \ie, from ground state $\ket{0}$ to excited state $\ket{20}$. Set $\omega = 1.0$ and initial condition is chosen as $\psi(0) = \ket{k}$ with probability $p_k\propto \text{Poissson}(k, 0.5)$ for $k=0,1,\cdots, 4$; hence the density matrix $\oprho(0)$ is again a rank-5 matrix. 
This model has been tested for two parameter sets $\gamma_1 = 0.2$, $\gamma_2 = 0$ and $\gamma_1 = 0$, $\gamma_2 = 0.2$. 
In the first case, 
the environment acts as an annihilation operator to the system so that the system is moving to a lower energy state, whereas in the second case, the environment acts as a creation operator so that the system is moving up to a higher energy state. 
The above \lb{} equation is solved by both QSD and LQSD unraveling schemes, as well as the \thismethodshort{} method for QSD and LQSD correspondingly (see  \secref{sec::discuss::unravel} for expressions of QSD and LQSD).  For this example, the solution of \lb{} equation via numerical ODE integrator is used as the reference and treated as the exact solution.

The sample size is $3\times 10^5$ and time step is $\frac{1}{500}$. The results are visualized in \figref{fig::qdho_test1} and \ref{fig::qdho_test2}. The spectrum is consistent with the physical intuition. In \figref{fig::qdho_test1}, the highest eigenvalue  roughly indicates the probability at ground state $\ket{0}$, which steadily increases. This is consistent with the functioning of annihilation operator, \ie, moving the quantum state to lower energy state. The spectrum in  \figref{fig::qdho_test2} can be explained in a similar way. 
When the rank increases from 3 to 5, relative error decreases in all figures. The solution at rank $5$ is comparable with the original unraveling SDE system (\ie, $r=21$). 
In the first case $\gamma_1 = 0.2$ and $\gamma_2 = 0.0$, since the system is lowering down to ground state, it is expected that low-rank approximation should work better.

\begin{figure}[h]
  \centering
  \parbox{\figrasterwd}{
    \parbox{.45\figrasterwd}{%
      \subcaptionbox{}{\includegraphics[width=\hsize]{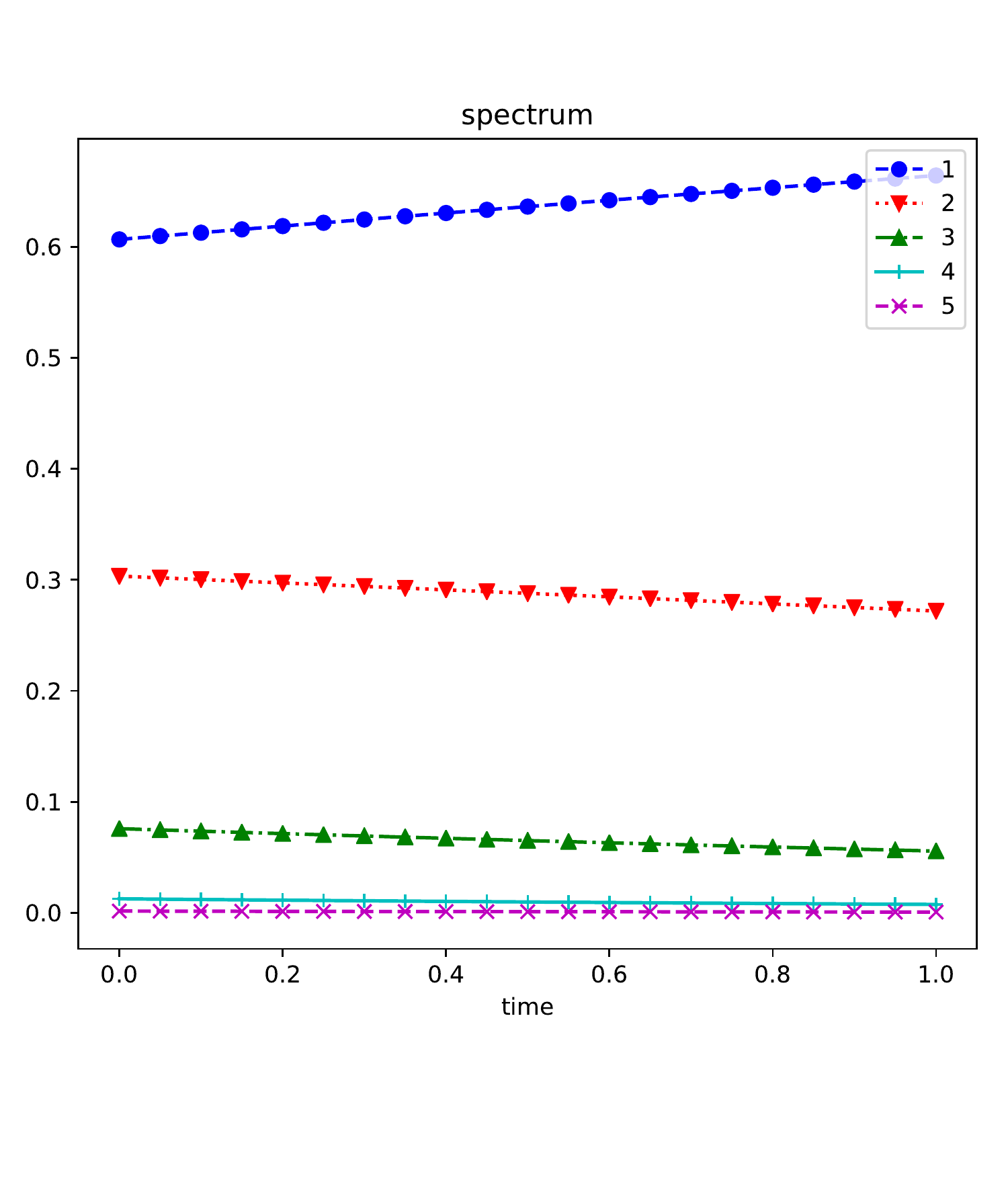}}
      }
    \parbox{.48\figrasterwd}{%
      {\includegraphics[width=\hsize]{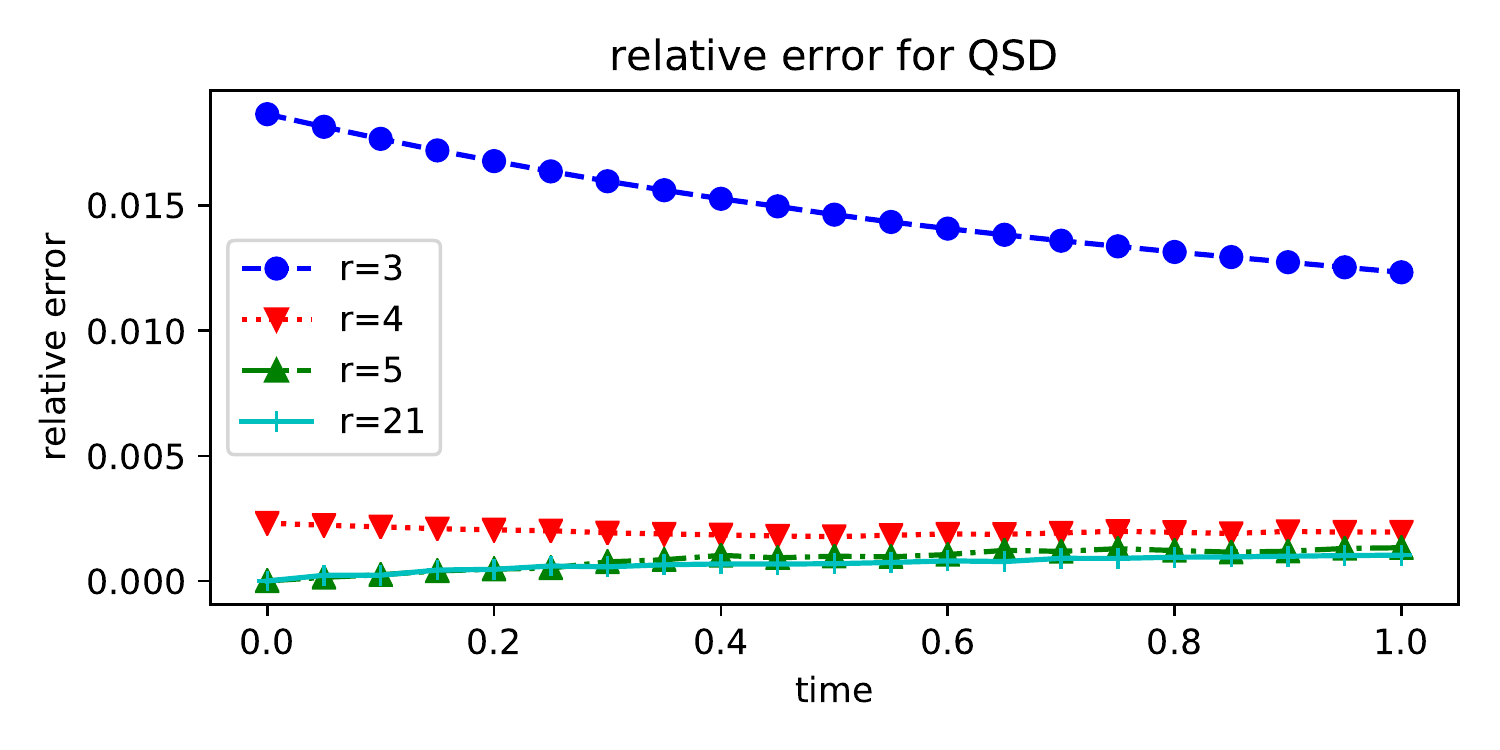}}
      {\includegraphics[width=\hsize]{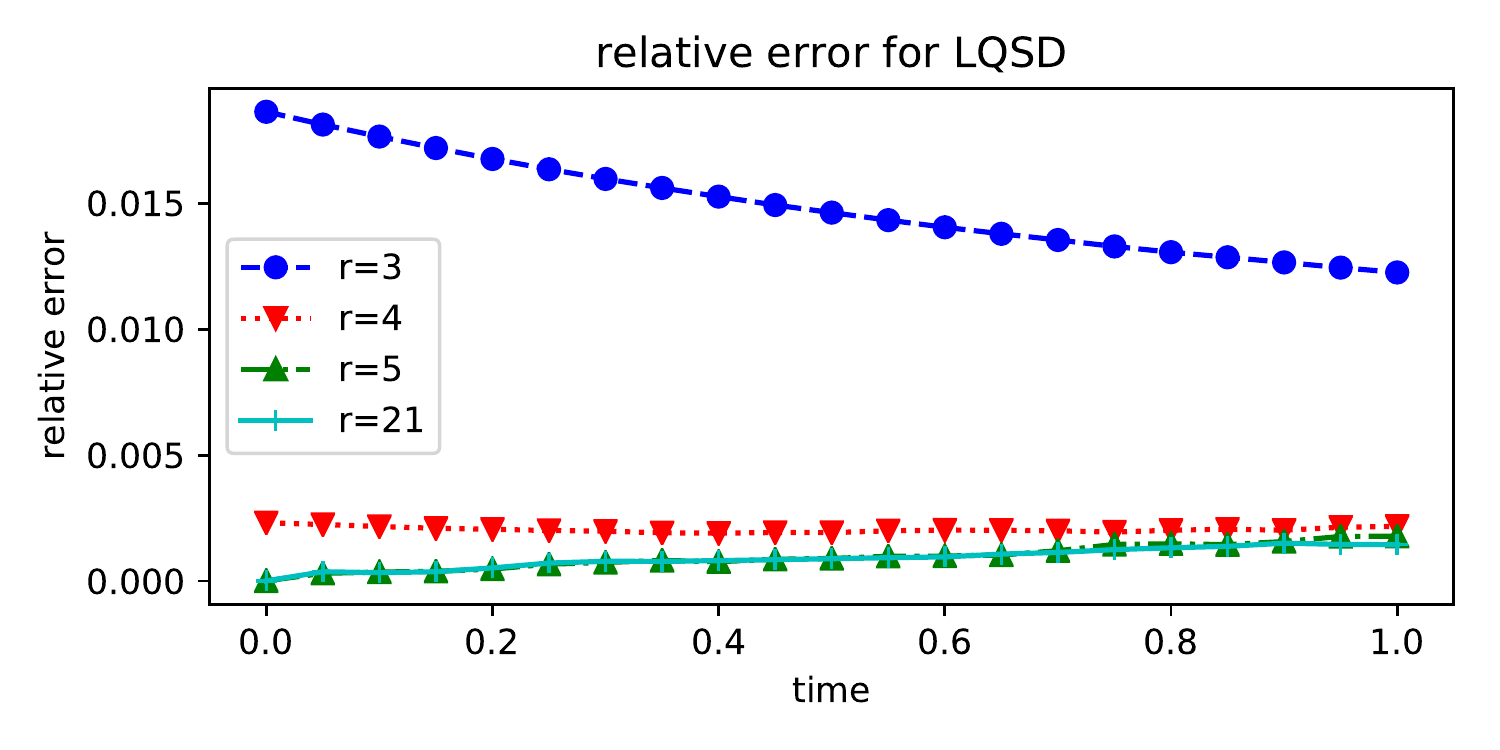}}  
    }
  }
  
  \caption{Spectrum (the five largest eigenvalues) and relative error of $\ee_{\mu_t} \bigl[xx^{\dagger}\bigr]$, for both QSD and LQSD unraveling schemes in quantum damped harmonic oscillator  with parameters $\gamma_1 = 0.2$, $\gamma_2 = 0.0$.}
\label{fig::qdho_test1}
\end{figure}

\begin{figure}[h]
  \centering
  \parbox{\figrasterwd}{
    \parbox{.45\figrasterwd}{%
      {\includegraphics[width=\hsize]{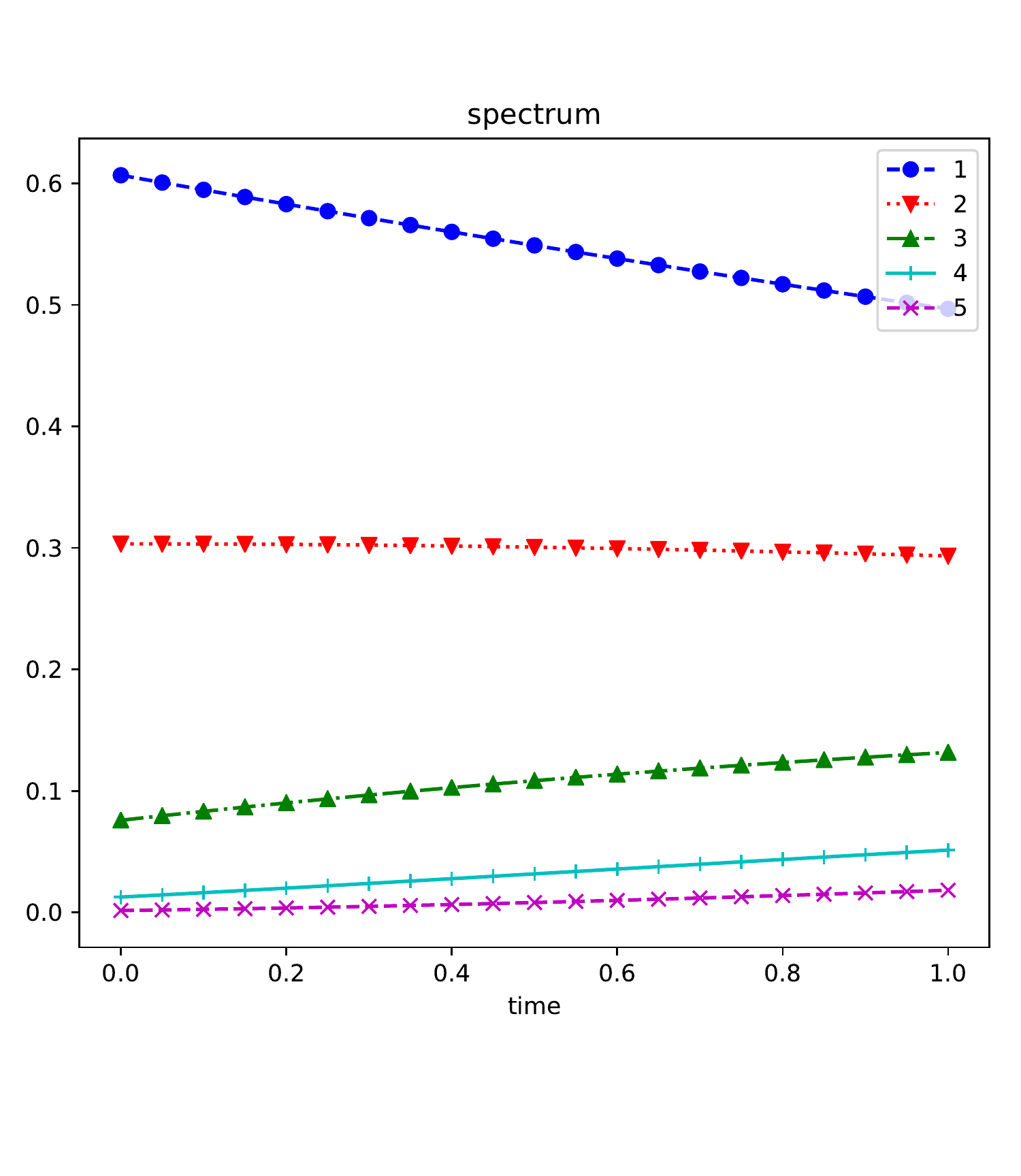}}
      }
    \parbox{.48\figrasterwd}{%
      {\includegraphics[width=\hsize]{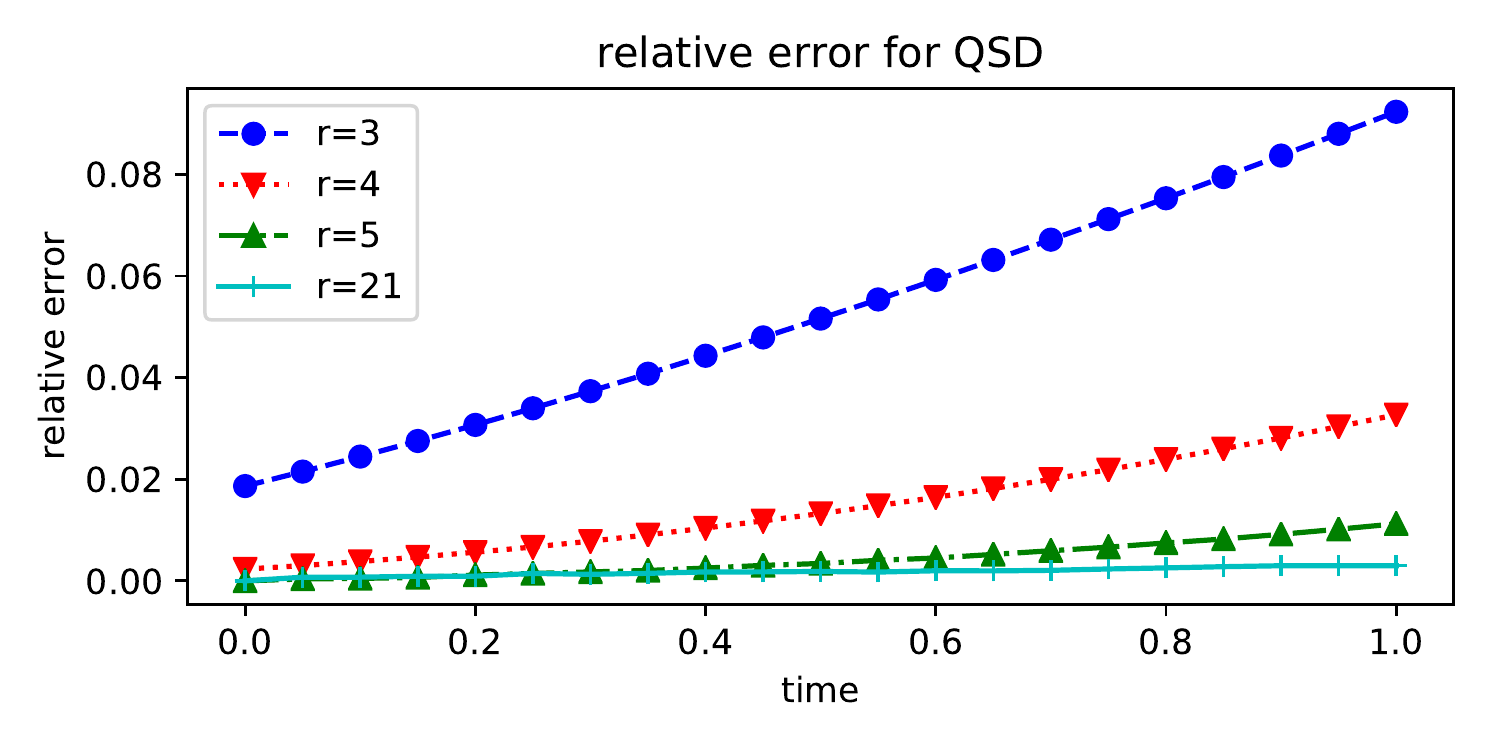}}
      {\includegraphics[width=\hsize]{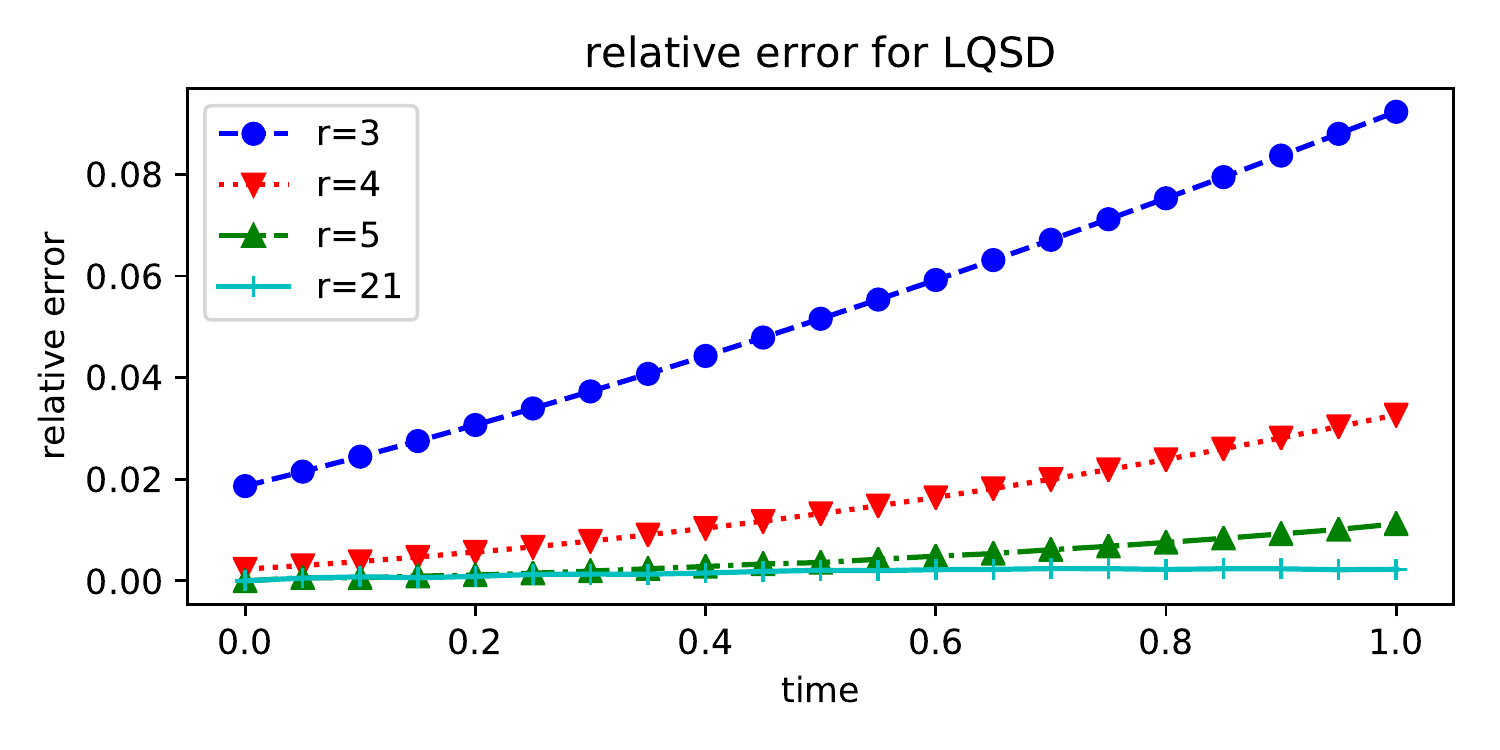}}  
    }
  }

\caption{Spectrum (the five largest eigenvalues) and relative error of $\ee_{\mu_t} \bigl[xx^{\dagger}\bigr]$, for both QSD and LQSD unraveling schemes  in quantum damped harmonic oscillator with parameters $\gamma_1= 0.0$, $\gamma_2 = 0.2$.
}
\label{fig::qdho_test2}
\end{figure}

\newpage
\section{Conclusions and outlook}
\label{sec::conclusion}

In this paper, we have proposed a tangent space projection method in
the space of finite signed measures; it is termed as \thismethod{}
method (or \thismethodshort{} in abbreviation).  Then by applying
\thismethodshort{} method, we have derived the low-rank dynamics of
SDE in \thmref{thm::fx=xx}, obtaining an ODE-SDE coupled system as a
low-rank approximation of the original high-dimensional SDE in the weak sense.
We have also established a commuting diagram for the action of stochastic
unraveling and dynamical low-rank approximation. This method has further been
validated by error analysis. Three numerical examples have been provided in
\secref{sec::numerics} to demonstrate the good performance of this
low-rank approximation method for dissipative physical systems.

There are some continuing interesting questions to explore. For
instance, whether \thismethodshort{} method can be extended to
infinite-dimensional Hilbert space instead of $\hbt$. As is
well-known, many interesting quantum master equations (as well as
their unraveling schemes) evolve on an infinite-dimensional Hilbert
space. Also, many SPDEs are essentially SDEs on infinite dimensional
Hilbert space, \eg, stochastic Burgers' equation in the form given in
last section.  Though finite truncation is a must in numerical
simulation in practice, it is still desirable to see whether our
scheme could be applied to such SPDEs directly in theory.  Another
interesting question is to develop adaptive schemes for
\thismethodshort{} method that automatically adjust the rank
on-the-fly.  We shall leave these questions to future works.

\section*{Acknowledgment}
This work is partially supported by the National Science Foundation
under award DMS-1454939.

\appendix
\section{Additional proofs}
\subsection{Tangent space of Stiefel manifold $\uspace$}
\label{sec::stiefel_tangent}

Stiefel manifold on $\hbt$ is a collection of $\dimn \times r$ complex-valued matrices with orthonormal columns; it is denoted by $\uspace$. 
For a differentiable trajectory $\opu(t)$ in the Stiefel manifold, one has $\opu^{\dagger}(t) \opu(t) = \id_{r\times r}$ for all time $t$, thus
$\dot{\opu}^{\dagger}(t) \opu(t) + \opu^{\dagger}(t) \dot{\opu}(t) = 0$
and consequently, tangent space
$\tang_{\opu(t)} \uspace = \left\{\opv \in \Complex^{\dimn\times r}: \opv^{\dagger} \opu(t) + \opu^{\dagger}(t) \opv = 0 \right\}.$
Then we shall prove the following lemma.

\begin{lemma} The tangent space of Stiefel manifold is given by
\[
\tang_{\opu(t)} \uspace = \left\{ i \opg \opu(t):\ \opg^{\dagger} = \opg \right\}.
\]
\end{lemma}

\begin{proof}
Let us first prove $\tang_{\opu(t)} \uspace \subset \left\{i \opg \opu(t):\ \opg \text{ is a Hermitian matrix} \right\}$. The other direction is trivial.

Suppose the columns of $\opu(t)$ are denoted by $U_1, U_2, \cdots U_r$ respectively. One could complete the basis and find an orthonormal set $\left\{U_{r+1}, \cdots U_{\dimn-1}, U_{\dimn}\right\}$. Let $\opu_{\perp} = \begin{bmatrix} U_{r+1} & \cdots &  U_{\dimn-1} & U_{\dimn}\end{bmatrix}$. Then $\begin{bmatrix} \opu(t) & \opu_{\perp}\end{bmatrix}$ is a unitary matrix. 
Let us extend the matrix $\opv\in \tang_{\opu(t)} \uspace$ 
to be $\begin{bmatrix} \opv & \wt{\opv} \end{bmatrix}$ 
and define matrix $\wt{\opg}$ by 
$\begin{bmatrix}\opv & \wt{\opv} \end{bmatrix} \begin{bmatrix} \opu(t) & \opu_{\perp}\end{bmatrix}^{\dagger} =: i \wt{\opg}$. The choice of $\wt{\opv}_{\dimn\times (\dimn-r)}$ does not play any role in the proof; it is only introduced to conveniently define $\wt{\opg}$. 
Then
\[i \wt{\opg} = \opv \opu^{\dagger}(t) + \wt{\opv} \opu_{\perp}^{\dagger}.\]
Then one could calculate that
\[\left\{\begin{split}
& \proj_{\opu(t)} \wt{\opg} \proj_{\opu(t)} = (-i) \opu(t) \left(\opu^{\dagger}(t) \opv \right) \opu^{\dagger}(t) \\
& \projperp_{\opu(t)}  \wt{\opg} \proj_{\opu(t)} = (-i) \projperp_{\opu(t)} \opv \opu^{\dagger}(t), 
\end{split}\right.\]
where $\proj_{\opu(t)} := \opu(t) \opu^{\dagger}(t)$ and 
$\projperp_{\opu(t)} := \id - \proj_{\opu(t)}$.
It could be straightforwardly verified that $\proj_{\opu(t)} \wt{\opg} \proj_{\opu(t)}$ is Hermitian,  
due to the assumption that 
$\opv \in \tang_{\opu(t)} \uspace$. 
Let us define Hermitian matrix $\opg$ by
\[\opg := \proj_{\opu(t)} \wt{\opg} \proj_{\opu(t)} + \projperp_{\opu(t)} \wt{\opg} \proj_{\opu(t)} + \proj_{\opu(t)} \wt{\opg}^{\dagger} \projperp_{\opu(t)} .\]
Then 
\[i \opg \opu(t) = \left( \opu(t) \bigl(\opu^{\dagger}(t) \opv \bigr) \opu^{\dagger}(t) + \projperp_{\opu(t)} \opv \opu^{\dagger}(t) \right) \opu(t) = \opv. \]
Thus, $\tang_{\opu(t)} \uspace \subset \left\{i \opg \opu(t):\ \opg \text{ is a Hermitian matrix} \right\}$ and the proof is completed.
\end{proof}

\subsection{Equivalent choices of test functions}\label{sec:equivalence}

\begin{lemma}
\label{lemma::choice_equivalent_F}
In defining \sdm{} in \eqrefn{eqn::metric}, the following two choices are equivalent,
\begin{enumerate}
\item $\funcsp = \left\{f(x) = xx^{\dagger}\right\}$ with Hilbert-Schmidt norm.

\item \[\funcsp = \left\{f(x)=\Inner{x}{\opo x} \;\big\vert \;
  \text{Hermitian matrix } \opo \text{ satisfies } \norm{\opo}_{\hs} \le
  1\right\}\]
and the associated norm is simply absolute value.
\end{enumerate}
\end{lemma}

\begin{proof}
Denote the \sdm{} defined by choice (1) as $d_{\funcsp_1}$ and the one defined by choice (2) as $d_{\funcsp_2}$. 

Firstly, for $\nu = \nu_1 - \nu_2 \in \mani$, 
\[\begin{split}
 \Abs{\int \Inner{x}{\opo x}\ \ud \nu} &= \Abs{\int \tr(xx^{\dagger} \opo)\ \ud \nu} = \Abs{\tr\left(\int xx^{\dagger}\ \ud\nu \opo \right)} \\
&\le \Norm{\int xx^{\dagger}\ \ud \nu}_{\hs} \Norm{\opo}_{\hs} \le \Norm{\int xx^{\dagger}\ \ud \nu}_{\hs}  = d_{\funcsp_1}(\nu_1, \nu_2).
\end{split}\]
Then, it implies that $d_{\funcsp_2}(\nu_1, \nu_2) \le d_{\funcsp_1}(\nu_1, \nu_2)$. 
Secondly, we shall prove that the equality could be reached. The equality is reached when $\opo = \frac{\int xx^{\dagger}\ \ud \nu}{\Norm{\int xx^{\dagger}\ \ud \nu}_{\hs}}$, which is also a Hermitian matrix. Thus these two choices are equivalent. 
 
\end{proof}

\subsection{Proof of \thmref{thm::lebris}}
\label{app::proof_lebris}
At fixed time $t$, the tangent space of $\oprho_{\lr}(t) \equiv \LU(t) \LSIG(t) \LU^{\dagger}(t) \in \mani_r$ is parametrized by Hermitian matrices $\opg(t)$ and $\opeta(t)$ with the form
\[\frac{\ud}{\ud t}{\oprho}_{\lr}(t) = i \opg(t) \LU(t) \LSIG(t) \LU^{\dagger}(t) + \LU(t) \opeta(t) \LU^{\dagger}(t) -i  \LU(t) \LSIG(t)  \LU^{\dagger}(t) \opg(t), \]
where $\opg(t) = \projperp_{\LU(t)} \opg(t) \proj_{\LU(t)} + \hc$, 
while $\opeta(t)$ is any Hermitian matrix.

Denote $\opc \equiv \lbop\bigl(\oprho_{\lr}(t)\bigr) - \left(i \opg(t) \LU(t) \LSIG(t) \LU^{\dagger}(t) + \LU(t) \opeta(t) \LU^{\dagger}(t)  - i  \LU(t) \LSIG(t)  \LU^{\dagger}(t) \opg(t)\right)$. Then we need to minimize $\Norm{\opc}_{\hs}^2 = \Inner{\opc}{\opc}_{\hs}$ for varying $\opg(t)$ and $\opeta(t)$. The first order stationary conditions with respect to $\opg(t)$ and $\opeta(t)$ yield
\[\left\{\begin{split}
0 &=\tr\left( \Comm{ \oprho_{\lr}(t)}{ \opc } \delta \opg \right), \\
0 &= \tr\left(\LU^{\dagger}(t) \opc \LU(t) \delta \opeta \right).\\
\end{split}\right.\]
Similar to the argument in the proof of \thmref{thm::fx=xx}, since $\delta \opg$ is any Hermitian matrix satisfying $\delta \opg = \projperp_{\LU(t)} \delta \opg \proj_{\LU(t)}+ \hc$, from the first part in the last equation, one could obtain
\[\begin{split}
0 &= \projperp_{\LU(t)} \Comm{ \oprho_{\lr}(t)}{ \opc}  \proj_{\LU(t)} \\
&= - \projperp_{\LU(t)} \opc \LU(t) \LSIG(t) \LU^{\dagger}(t). \\
\end{split}\]
After replacing $\opc$ by its definition and some simplification, one could obtain
\[\projperp_{\LU(t)} \opg(t) \proj_{\LU(t)} = -i \projperp_{\LU(t)} \lbop\bigl( \oprho_{\lr}(t) \bigr) \LU(t) \LSIG(t)^{-1} \LU^{\dagger}(t).\]
Then by the fact that $\frac{\ud}{\ud t} \LU(t) = i \opg(t) \LU(t)$, the time-evolution equation for $\frac{\ud}{\ud t}\LU(t)$ could be derived easily. 

From the second part of first order stationary condition (\ie, with respect to $\opeta(t)$), one could deduce that $\LU^{\dagger}(t) \opc \LU(t) = 0$. 
After plugging in the expression of $\opc$, 
it follows immediately that
\[\opeta(t) = \LU^{\dagger}(t) \lbop\big(\oprho_{\lr}(t) \bigr)\LU(t).\]

\addchange{
Since there is only one solution satisfying stationary conditions,
the optimal solution for pair $(\opg(t), \opeta(t))$ is unique.}

\bibliographystyle{elsarticle-num}
\bibliography{ref.bib}

\end{document}